\documentclass{article}
\usepackage[english]{babel}
\usepackage[utf8x]{inputenc}
\usepackage[colorinlistoftodos]{todonotes}
\usepackage{amsmath,amsfonts,amssymb,hyperref,color,graphicx,siunitx,fancyhdr,lastpage,enumerate, commath, amsthm}
\usepackage{url}
\usepackage[authoryear]{natbib}

\usepackage{algorithm}
\usepackage[noend]{algpseudocode}
\usepackage{appendix}

\newtheorem{theorem}{Theorem}
\newtheorem{lemma}[theorem]{Lemma}
\newtheorem{proposition}[theorem]{Proposition}
\newtheorem{assumption}[theorem]{Assumption}

\newtheorem{remark}[theorem]{Remark}

\newcommand{\R}{\mathbb R}
\newcommand{\Rd}{\mathbb R^d}
\newcommand{\E}{\mathbb E}
\renewcommand{\P}{\mathbb P}
\renewcommand{\L} {\mathcal L}
\newcommand{\Lk}{L^\kappa}

\newcommand{\Dom}{\mathcal D}
\newcommand{\taud}{\tau_\partial}

\newcommand{\ubar}[1]{\text{\b{$#1$}}}
\newcommand{\Res}{\mathcal R}

\newcommand{\etLk}{\exp(-t\Lk)}
\newcommand{\esLk}{\exp(-s\Lk)}
\newcommand{\kapMe}{\kappa_M^{\textrm e}}

\newcommand{\nut}{\nu_{\pi}}
\newcommand{\nur}{\nu_{\mu}}
\newtheorem{contribution}{Contribution}

\newcommand{\footremember}[2]{%
    \footnote{#2}
    \newcounter{#1}
    \setcounter{#1}{\value{footnote}}%
}
\newcommand{\footrecall}[1]{%
    \footnotemark[\value{#1}]%
} 

\usepackage{todonotes}

\title{Regeneration-enriched Markov processes with application to Monte Carlo}  
\author{
Andi Q. Wang\footremember{Bristol}{University of Bristol, Department of Mathematics}\footnote{Corresponding author. Address: Fry Building, Woodland Road, Bristol, BS8 1UG, UK. Email: \href{mailto:andi.wang@bristol.ac.uk}{andi.wang@bristol.ac.uk}.}
\and Murray Pollock\footremember{New}{Newcastle University, School of Mathematics, Statistics and Physics}\footremember{ATI}{The Alan Turing Institute}
\and Gareth O. Roberts\footremember{War}{University of Warwick, Department of Statistics}\footrecall{ATI}
\and David Steinsaltz\footnote{University of Oxford, Department of Statistics}}

\date{June 2020}

\begin{document}

\maketitle

\begin{abstract} 
We study a class of Markov processes that combine local 
dynamics, arising from a fixed Markov process, 
with regenerations arising at a 
state-dependent rate. 
We give conditions under which such processes possess a given target distribution as their invariant measures, thus making them amenable for use within Monte Carlo methodologies. 
Since the regeneration mechanism can compensate the choice of local dynamics, while retaining the same invariant distribution, great flexibility can be achieved in selecting local dynamics, and the mathematical analysis is simplified.
We give straightforward conditions for the process to possess a central limit theorem, and additional conditions for uniform ergodicity and for a coupling from the past construction to hold, enabling exact sampling from the invariant distribution. We further consider and analyse a natural approximation of the process which may arise in the practical simulation of some classes of continuous-time dynamics. 

\end{abstract}

\section{Introduction}
In this work we study a broad class of continuous-time Markov processes $X$ which are defined by
superimposing regenerative dynamics 
on to an existing continuous-time Markov process $Y$ on state space $E$. 
The precise definition of the process is given in Section \ref{sec:construction}, but $X$ can be seen informally as a Markov process with infinitesimal generator $L^\mu$ given by
\begin{align}
     L^\mu f(x) = Q f(x) + \kappa(x) \int \left( f(y) - f(x) \right)\mu(y)m(\dif y),     \label{eq:gen}
\end{align}
where $Q$ denotes the infinitesimal generator of the process $Y$.

We will refer to the function $\kappa$ as the \textit{regeneration rate} and to $\mu$ as the \textit{regeneration density}. Collectively, $\kappa$ and $\mu$ constitute the \textit{global dynamics}.
We will refer to the dynamics defined by the process $Y$ as the \textit{local dynamics}. For example, we may choose the local dynamics to be a Brownian motion on $\Rd$ or a continuous-time jump process, such as a suitably-defined Metropolis--Hastings chain embedded in continuous time.

Fundamental to the introduction of this class of Markov processes -- which we term {\em Restore} processes as they are {\em Randomly Exploring and  STOchastically REgenerating} -- is that the global regenerative dynamics we introduce can enrich the existing local dynamics of $Y$ in a compensatory manner, such that the invariant distribution of $X$ is known. This is of particular application within Monte Carlo methodology as the dynamics can often be chosen, and the process straightforwardly simulated, such that its invariant distribution coincides with a prescribed target density of interest, $\pi$.

This incorporation of global regenerative dynamics to enrich an existing continuous-time Markov process $Y$ introduces a number of directions for both theoretical and methodological innovation, which we explore in this paper. 
Since the process naturally exhibits regenerations, mathematical analysis of the Restore process is simplified considerably. Indeed, the traditional approach of analysing Markov chains, in the sense of \cite{Meyn1993a}, crucially relies on the identification of regeneration times.

Thus the Restore process forms the basis of a new approach to Monte Carlo sampling, which we term the \textit{Restore sampler}.
The Restore sampler also provides a simple recipe for introducing (nonreversible) rejection-free moves to existing samplers. This can be done in cases where standard Markov chain Monte Carlo (MCMC) algorithms may exhibit poor mixing. This is discussed in Sections~\ref{subsec:jump_proc}. 

The Restore process is an instance of a `resurrected' or `returned' process, which instantaneously returns to the state space after being killed. Such processes have been utilized extensively within probability literature. Their use goes back to the very foundations of Markov chain theory, \cite{Doob1945}, but such processes have been harnessed particularly effectively in the study of \textit{quasi-stationarity}. See for instance, \cite[Section 3.4]{Bartlett1960}, \cite{Darroch1965}, \cite[Chapter 4.4]{Collet2013}, \cite{Barbour2010, Barbour2012}, \cite{Benaim2018}, \cite{wang2019approx}, \cite{Wang18Note}. 
For example, such processes have been to used to approximate quasi-stationary distributions, and in this context, \cite{Darroch1965} noted that for discrete-time, finite state space resurrected processes, the invariant distribution could be `made into almost any distribution'. The work of this paper demonstrates that for continuous time and general state spaces, this is also the case.

The idea of identifying regeneration times within a given MCMC sampler goes back to \cite{Mykland1995}, using the very elegant splitting technique of \cite{Nummelin1978}. The area has continued to develop actively, as seen for instance in the contributions of \cite{gilks1998adaptive,Hobert2002,
Brockwell2005,
Minh2012, Lee2014}.
The idea of hybridising separate dynamics has also had a long history, see, for instance, \cite{Tierney1996, Murdoch1998, Murdoch2000}, although these typically involve combining separate MCMC chains which are already themselves $\pi$-invariant. The Restore process offers practitioners considerable scope to design highly optimised sampling algorithms due to the flexibility of being able to `hybridise' dynamics which are separately not $\pi$-invariant.

Unlike traditional MCMC methods, the Restore sampler is a fundamentally continuous-time sampler, as the inhomogeneous Poisson clock dictating the regeneration events is crucial for aligning the local and global dynamics. 
In a similar vein, the class of \textit{piecewise-deterministic Markov processes} (PDMPs, \cite{Davis1984}), and \textit{quasi-stationary Monte Carlo methods} (QSMC) also make use of an inhomogeneous Poisson process to drive the process towards the target distribution $\pi$; see \cite{Vanetti2017, Wang2019}.
Notable examples of such methods include the Bouncy Particle Sampler, \cite{Bouchard-Cote2018}, the Zig-Zag Sampler, \cite{Bierkens2019}, ScaLE, \cite{Pollock2016} and ReScaLE, \cite{Kumar2019}. 

Sampling algorithms which rely upon continuous-time dynamics often require some form of approximation for their practical implementation, and the resulting approximate process can exhibit algorithmic instability, or possess an approximate invariant distribution which is intractable. However, there is considerable scope and promise to understand the effect of such approximations with the Restore process, due to the global regenerative dynamics with which it is constructed and the ease with which it can be mathematically analysed (for instance, in the sense of \cite{Asmussen2007}). In Section~\ref{s:truncated}, we consider one natural approximation to the Restore process in which the regeneration rate is truncated. 

\subsection{Summary of results}

We begin in Section~\ref{sec:construction} by formally introducing the Restore process on an abstract state space $E$, and in Section~\ref{sec:invariance} we will establish the following.

\begin{contribution}[$\pi$-invariance:  Theorems~\ref{thm:restore_invariant_diffusion},~\ref{thm:jump_invar}] 
    Assume that we are given a positive target density $\pi$ on $E$, a regeneration density $\mu$ with on $E$, and an interarrival process $Y$ with infinitesimal generator $Q$ with adjoint $Q^*$. 
We assume that we have chosen a constant $C>0$ such that
\begin{equation}
    Q^*\pi +C\mu \ge 0.
    \label{eq:nonnong_cond}
\end{equation}

    Under a range of settings and regularity conditions, to be detailed in Section \ref{sec:invariance}, the resulting Restore process with interarrival dynamics $Y$, regeneration rate $\kappa$ and regeneration density $\mu$ has invariant density $\pi$.
\end{contribution}
We consider the following two indicative settings: symmetric diffusion processes and continuous-time jump processes.

In Section~\ref{sec:limiting_properties}, we study limiting properties of the Restore process and will present the following results.

\begin{contribution}[Central Limit Theorem: Theorem~\ref{thm:CLT}] 
    Writing $(T_n)$ for the regeneration times, then for appropriate functions $f$, where $\sigma_f^2$ is the asymptotic variance defined in \eqref{eq:rest:var_f}, then under appropriate regularity conditions the following holds for the Restore process:
         \begin{equation*}
         \sqrt n \left ( \frac{\int_0^{T_n} f(X_s)\dif s}{T_n} - \pi[f] \right) \overset{d}\to N\left (0, \sigma_f^2\right).
     \end{equation*}
\end{contribution}

Under additional assumptions, we will derive uniform ergodicity and a \textit{coupling from the past} (CFTP) construction (following \cite{Propp1996}), which is particularly useful in the context of Monte Carlo simulation, since it allows us to obtain an exact draw from the target $\pi$.

\begin{contribution}[Uniform ergodicity, CFTP: Proposition~\ref{prop:unif_erg}, Theorem~\ref{thm:CFTP}] 
    Assume that the regeneration rate $\kappa$ is uniformly bounded away from 0 and basic regularity conditions hold. 
    Then 
the Restore process is uniformly ergodic. Furthermore, there is a straightforward coupling from the past construction.
\end{contribution}
Indeed, in Theorem~\ref{thm:rej_samp}, we show that the classical rejection sampler is a special case of this coupling from the past construction.

In Section~\ref{sec:simulation}, we discuss some practical considerations related to the Restore sampler, and in particular we present a result concerning the error incurred when running one natural approximation of the Restore process. 

\begin{contribution}[Truncated rate: Theorem~\ref{Thm:bias_M}, Proposition~\ref{prop:trunc_bd}]
    When the interarrival process is a diffusion and $\kappa$ is bounded away from 0, consider running the Restore process with a \textit{truncated} version of the regeneration rate $\kappa_M$: 
    \begin{equation*}
        \kappa_M := \kappa \wedge M.
    \end{equation*}
    Writing $\pi_M$ for the invariant distribution of the resulting approximate process, we provide a bound on the error $\|\pi_M-\pi\|_{1}$ in total variation and show it vanishes to 0 as the truncation level $M\to \infty$.
\end{contribution}

Some simple examples highlighting various aspects of the Restore sampler are given in Section~\ref{sec:Restore_examples}. To conclude, in Section~\ref{sec:Restore_conclusions} we discuss the limitations of our approach, and possible future directions. Some technical proofs are omitted from the body of the text for readability, but can found in the Appendices.

\section{The Restore process}
\label{sec:construction}
First we formalize the informal definition of the Restore process given in the introduction. We define the process in a general, abstract framework.

Let $(E,m)$ be a measure space, where $E$ is a Radon topological space with its Borel $\sigma$-algebra $\mathcal E$ and $m$ is a $\sigma$-finite Radon measure on $\mathcal E$, for example $\Rd$ equipped with Lebesgue measure. 
We assume that we are given a \textit{right process} $Y=(\Omega', \mathcal F', \mathcal F_t',Y_t, \P_x^0)$ evolving on $E$. 
Right processes are an abstract class of right-continuous strong Markov processes.
We do not repeat their precise definition, which is highly technical, here; the interested reader is referred to \cite[Chapter 20]{Sharpe1988}, instead we give a list of examples in the following lemma.

\begin{lemma}
        The following processes are examples of right processes: 
 deterministic right-continuous flows, 
Feller processes, 
Markov jump processes.
\end{lemma}
\begin{proof}
    See \cite[Exercise 8.8]{Sharpe1988}, \cite[Exercise 9.27]{Sharpe1988} and \cite[Exercise 14.18]{Sharpe1988}.
\end{proof}

\begin{remark}
    Recall that a Feller process is Markov process on a locally compact, Hausdorff, second countable space $E$, whose semigroup $(P_t)$ is strongly continuous on $C_0(E)$, the set of continuous functions vanishing at infinity. Examples of Feller processes include L\'evy processes, \cite[p50]{Sharpe1988}, and diffusions such as the ones studied in \cite[Chapter 1]{Demuth2000}.
    \label{rmk:Feller}
\end{remark}

For a general initial distribution $\nu$ we write $\P^0_\nu = \int_E \nu(\dif x)\,\P_x^0$.
Let $\kappa: E \to \R^+ = [0, \infty)$ be a locally bounded measurable function, the \textit{regeneration rate}. Define the \textit{lifetime} $\taud$ as
\begin{equation}
    \tau_\partial := \inf\bigg\{t\ge0: \int_0^t \kappa(Y_s)\dif s\ge \xi\bigg \},
    \label{eq:taud}
\end{equation}
where $\xi\sim \text{Exp}(1)$, independent of $Y$. Set $\inf \varnothing = \infty$.

Fix a probability measure $\nur$ on $(E,m)$, the \textit{regeneration distribution}. 
We define the Restore process $X=(X_t)_{t \ge 0}$ to be the process given by
\begin{equation}
    X_t = \sum_{i=0}^\infty  1_{[T_{i}, T_{i+1})}(t)\, Y^{(i)}_{t - T_{i}},
    \label{eq:Restore_process}
\end{equation}
where $(Y^{(0)}, \tau^{(0)})$ is a realisation of $(Y, \taud)$ with $Y_0 = x$, and $(Y^{(i)}, \tau^{(i)})_{i=1}^\infty$ are i.i.d. realisations of $(Y, \taud)$ under $\P^0_{\nur}$, namely with $Y_0 \sim \nur$.
The $(T_i)_{i=0}^\infty$ are given by $T_0 = 0$, and 
$  T_n = \sum_{i=0}^{n-1} \tau^{(i)}$, for each $n=1,2,\dots$. 

This defines a Markov process $X=(\Omega, \mathcal F, \mathcal F_t, X_t, \P_x)$ with state space $(E,m)$. For an arbitrary initial distribution $\nu$, as usual we set $\P_\nu = \int \dif \nu(x) \P_x$. In future, the regeneration measure $\nur$ will be given by a density function $\mu$ with respect to the reference measure $m$, and hence for its semigroup we will write $\{P_t^\mu: t \ge 0\}$. 
We will then refer to this process as the \textit{Restore process with interarrival dynamics $Y$, regeneration rate $\kappa$, and regeneration density $\mu$}. 

\begin{lemma}
    \sloppy Let $Y$ be a right process on the Radon space $(E,m)$ with Radon measure $m$, $\kappa:E\to \R^+$ a locally bounded measurable function, and $\mu$ a probability measure on $E$. Then the resulting Restore process $X=(\Omega, \mathcal F, \mathcal F_t, X_t, \P_x)$ with interarrival dynamics $Y$, locally bounded nonnegative regeneration rate $\kappa$ and regeneration density $\mu$ defines a right process with state space $(E,m)$. In particular, $X$ is right-continuous and strong Markov. Moreover, $T_n \to \infty$ almost surely.
    \label{prop:Restore_exist}
\end{lemma}
\fussy
\begin{proof}
    See Appendix~\ref{subsec:pf_exist}.
\end{proof}

\section{Invariance}
\label{sec:invariance}
Suppose we are given a probability measure $\nut$ on $E$, our target measure of interest.
We will assume throughout that the target measure $\nut$ and regeneration measure $\nur$ are given by density functions $\pi,\mu$ respectively with respect to the reference measure $m$, namely
\begin{equation*}
    \nut(\dif x)=\pi(x)m(\dif x),\quad \nur(\dif x)=\mu(x)m(\dif x).
\end{equation*}
We would like to construct a Restore process $X$ whose invariant distribution coincides with $\pi$. In this section we 
formulate conditions in several settings under which this is possible.

We consider the following settings: when the interarrival process $Y$ is a symmetric diffusion, and when the interarrival process is a jump process. This latter situation includes, for example, the case when the state space $E$ is countable.

Writing $Q$ for the generator of the process $Y$, define the regeneration rate $\kappa: E\to\R$ by
    \begin{equation}
        \kappa(x) := \frac{Q^* \pi(x)}{\pi (x)} + C \frac{\mu(x)}{\pi(x)}, \quad x \in E.
        \label{eq:kappa_generic}
    \end{equation}
We will make rigorous sense of this expression in the subsequent sections. 
\begin{remark}
    Because of the flexibility provided by the constant $C$ in \eqref{eq:kappa_generic}, in practice we do not require $\mu$ or $\pi$ to be normalized in order to compute $\kappa$.
\end{remark}

Given the formal generator \eqref{eq:gen}, we can make intuitive sense of the expression \eqref{eq:kappa_generic} from the following formal manipulations:
\begin{align*}
    \nu_\pi Qf = \int \pi(x) (Qf)(x) m(\dif x) = \int (Q^* \pi)(x) f(x)m(\dif x).
\end{align*}
Taking $f\equiv 1$ the constant function, we see that $\int Q^* \pi(x) m(\dif x)=0$, since $Q1\equiv 0$. Then,
\begin{align*}
    \nu_\pi[\kappa (\nu_\mu[f]-f)] &=\nu_\pi \left [\pi^{-1} (Q^*\pi +C\mu)\right] \nu_\mu[f]-\nu_\pi[\pi^{-1} f Q^*\pi] - C\nu_\pi[\pi^{-1}\mu f] \\
    &= 0 -m[f Q^*\pi] + C\left( m[\mu f]m[\mu] - m[\mu f] \right).
\end{align*}
This final bracket is 0 since $m[\mu]=1$, as $\nu_\mu$ is a probability measure. This allows us to conclude that
\begin{equation*}
    \nu_\pi [L^\mu f] = \nu_\pi [Qf]-\nu_\pi[\kappa (\nu_\mu[f]-f)]=m[\pi Qf]-m[f Q^*\pi]=0.
\end{equation*}
This calculation shows that our $\kappa$ is indeed of the right form to ensure invariance of $\nu_\pi$.

We emphasize again that the preceding calculations are formal and do not constitute a rigorous proof. In order to turn this into a full proof, one must first show that the operator $L^\mu$ given in \eqref{eq:gen} is indeed the generator of the Restore process (as constructed in Section~\ref{sec:construction}), carefully noting the domain $\mathcal D(L^\mu)$. We must then establish that the above calculations hold for a collection of functions $f \in D$, and prove that $D$ constitutes a core of the generator.

\begin{remark}
    Turning these calculations into a proof in a general setting is difficult for several reasons. First, establishing that $L^\mu$ is the generator of the Restore process is complicated since $\kappa$ is not necessarily bounded, thus the Restore process is not necessarily Feller in the sense of Remark~\ref{rmk:Feller}. This prevents us from straightforwardly establishing dissipativity, via the positive maximum principle, which would enable the application of general reformulations of the Hille--Yosida theorem such as Theorem~7.1 of \cite{Ethier1986}. 
    Second, proving that a collection of functions $D$ constitute a core for the generator is generally challenging. For recent advances on this topic for PDMPs, see the work of \cite{Durmus2018}.
\end{remark}
These difficulties associated with working in a general operator-theoretic setting are our motivation for considering our two specific settings separately; the diffusion setting, and the jump process setting. 
Indeed, one of the key contributions of this work is that in each setting we will give a proof of invariance which avoids using the full generator approach and the highly technical difficulties outlined above.

Once invariance is established, in order to approximate integrals we can make use of the following result. Recall that a nonnegative random variable is \textit{non-lattice} if it is not concentrated on a set of the form $\{\delta, 2\delta, \dots\}$ for any $\delta >0$.
\begin{theorem}
    Suppose that the Restore process $X$, as in the conclusion of Lemma~\ref{prop:Restore_exist}, is defined on a metric space $E$, its semigroup $P^\mu_t$ maps continuous functions to continuous functions for each $t \ge 0$, has a unique stationary distribution $\pi$, that $\E_\mu[\taud]<\infty$, and that the lifetimes are non-lattice. Then for any bounded measurable function $f: E \to \R$, we have that
    \begin{equation}
        \nu_\pi[f] = \frac{\E_\mu [\int_0^{\tau^{(0)}}f(X_s)\dif s]}{\E_\mu[\tau^{(0)}]},
        \label{eq:pi_f}
    \end{equation}
    and furthermore we have
    almost sure convergence of the ergodic averages: as $t \to \infty$,
    \begin{equation*}
        \frac{1}{t}\int_0^t f(X_s)\dif s \to \nu_\pi[f].
    \end{equation*}
\end{theorem}

\begin{proof}
    By Theorem~1.2 of \cite[Chapter 6]{Asmussen2003}, and uniqueness of the stationary distribution, it follows that \eqref{eq:pi_f} holds.
    Convergence of the ergodic averages then follows from the following arguments from renewal theory: First split $f$ into positive and negative parts, so we may assume that $f$ is nonnegative. Writing $(N(t))_{t\ge 0}$ for the renewal process of complete lfietimes before time $t$, we may thus bound
    \begin{equation*}
        \int^{T_{N(t)}}_0 f(X_s)\dif s \le \int^{t}_0 f(X_s)\dif s \le \int^{T_{N(t)+1}}_0 f(X_s)\dif s.
    \end{equation*}
    By the strong law of large numbers for renewal processes, Theorem~1 of \cite[10.2]{Grimmett2001}, we know that $N(t)/t \to 1/\E_\mu[\tau^{(0)}]$ almost surely. 
    We can conclude the argument by then applying the strong law of large numbers to $\int^{T_{N(t)}}_0 f(X_s)\dif s/N(t)$ and similarly for the upper bound.
\end{proof}

\subsection{Symmetric diffusions}
\label{subsec:symm_diff}
We first consider Restore when the underlying process is a symmetric diffusion on $E=\Rd$. For a smooth $C^\infty$ function $A:\Rd \to \R$ consider the stochastic differential equation (SDE)
\begin{equation}
    \dif Y_t = \nabla A(Y_t) \dif t + \dif B_t, \quad Y_0 = x,
    \label{eq:SDE}
\end{equation}
on $\Rd$ where $B$ is a standard Brownian motion on $\Rd$. Define the smooth function $\gamma: \Rd \to \R$ by
\begin{equation*}
    \gamma(y) = \exp(2A(y)), \quad y \in \Rd,
\end{equation*}
and define a measure $\Gamma$ on $\Rd$ by
\begin{equation*}
    \dif \Gamma(y) = \gamma(y) \dif y,
\end{equation*}
where $\dif y$ denotes Lebesgue measure on $\Rd$. 

    We are thus working on ${(E,m)=(\Rd, \Gamma)}$. This is an example of a Radon space with a Radon measure.

\begin{assumption}[Underlying process]
    $A:\Rd \to \R$ is a smooth $C^\infty$ function, and the SDE \eqref{eq:SDE} has a unique weak solution. The process $Y$ has a continuous symmetric transition density $p^0(t,x,y)$ on $(0,\infty)\times \Rd \times \Rd$ with respect to $\Gamma$, which satisfies the BASSA conditions of \cite[Chapter 1.B]{Demuth2000}. In particular, the diffusion is Feller, hence a right process.
    \label{assump:symmetric_markov_bassa}
\end{assumption}
The BASSA conditions of \cite[Chapter 1.B]{Demuth2000} are technical, and in Section~\ref{subsubsec:BASSA_ex} we will give examples of diffusions satisfying them.

The semigroup of the diffusion $Y$ is given for each $t\ge0$ by
\begin{equation}
    \E_x^0[f(Y_t)] = \int p^0(t,x,y) f(y) \dif \Gamma(y),
    \label{eq:semigp_Y0}
\end{equation}
for functions $f$ where this integral makes sense. Under Assumption \ref{assump:symmetric_markov_bassa}, the semigroup \eqref{eq:semigp_Y0} maps $C_0(\Rd)$ --- continuous functions vanishing at $\infty$ ---  into $C_0(\Rd)$ and is strongly continuous on $C_0(\Rd)$ with generator $-L^0$. Hence we can also write the semigroup as 
\begin{equation*}
    \E_x^0[f(Y_t)] = [\exp(-t L^0)f](x).
\end{equation*}
The action of the generator on smooth compactly supported $f$ is given by
\begin{equation*}
    -L^0 f = \frac{1}{2}\Delta f + \nabla A \cdot \nabla f.
\end{equation*}
Note that we are writing $L^0$ for \textit{minus} the generator, as is done in \cite{Demuth2000}.

Under Assumption~\ref{assump:symmetric_markov_bassa}, the semigroup is also strongly continuous on $$\mathcal L^p(\Gamma) := \left\{f:\R^d\to \R \text{ measurable}, \int_{\R^d}|f(x)|^p \dif \Gamma(x)<\infty\right\},$$
for each $1\le p <\infty$. When we want to emphasize the underlying function space we may write $-L^0_p$ for the corresponding generators on $\L^p(\Gamma)$ and ${\mathcal D(L^0_p)\subset \L^p(\Gamma)}$ for their respective dense domains. 

We now assume that the target distribution and regeneration distributions are defined by density functions with respect to $\Gamma$ denoted $\pi ,\mu \in \L^1(\Gamma)$ respectively:

\begin{assumption}[Densities]
    The target density $\pi \in \mathcal L^1(\Gamma)$, is positive on $\Rd$ and is twice continuously differentiable with $\int \pi \dif \Gamma =1$. The regeneration density $\mu$ is in $\L^1(\Gamma)$ and is nonnegative, with $\int \mu \dif \Gamma =1$. Furthermore, $\pi$ and
    $\mu$ are square-integrable --- that is, in $\L^2(\Gamma)$ --- 
    and $\pi$ is in the domain $\mathcal D(L^0_2)$.
    \label{assump:pi}
\end{assumption}
\begin{remark}
    Let us emphasize that we are writing $\pi$ and $\mu$ for densities with respect to the measure $\Gamma$, which may not necessarily be Lebesgue measure. Later on we will write $\bar \pi := \pi \gamma$ for the density with respect to Lebesgue measure.
\end{remark}
For our proofs we take $\pi, \mu$ to be normalized, but as noted previously this condition is not required in practice, because of the constant $C$ which appears in the regeneration rate.

Because $L^0$ is a self-adjoint operator on $\L^2(\Gamma)$, a sufficient condition for $\pi \in \Dom(L^0_2)$ is that $L^0 \pi \in \L^2(\Gamma)$. This is a well-known result; for a derivation, see, for example, \cite[Section 3.3.3]{Wang2020}, where a preliminary version of this work can also be found.

We can now define the regeneration rate $\kappa$, under Assumption~\ref{assump:pi}. First, define the \textit{partial regeneration rate} $\tilde \kappa$, via
\begin{equation*}
    \tilde \kappa(x) := \frac{1}{\pi (x)}\left(\frac{1}{2}\Delta \pi(x) + \nabla A \cdot \nabla \pi(x)\right), \quad x \in \R^d.
\end{equation*}
We define the actual regeneration rate $\kappa$ as follows.
Set for a given constant $C>0$,
\begin{equation}
    \kappa(x) := \tilde \kappa(x)  + C \frac{\mu(x)}{\pi(x)}, \quad x \in \Rd.
    \label{eq:kappa_symm}
\end{equation}

\begin{remark}
    Similarly to \cite{Wang2019}, writing $U := - \log \pi$, an equivalent expression for $\tilde \kappa$ is
    \begin{equation}
        \tilde \kappa(x) = \frac{1}{2}(-\Delta U(x) + |\nabla U(x)|^2) - \nabla A \cdot \nabla U(x).
        \label{eq:kappa_tilde}
    \end{equation}
\end{remark}

\begin{assumption}[Regeneration rate]
    The function $\kappa$ is continuous, and $C$ is chosen such that $\kappa \ge 0$.
    \label{assump:kappa}
\end{assumption} 
Under Assumptions \ref{assump:symmetric_markov_bassa}, \ref{assump:pi}, \ref{assump:kappa}, the process $Y$ killed at rate $\kappa$, that is, with lifetime given by $\eqref{eq:taud}$, can be analysed using Theorem 2.5 of \cite{Demuth2000}. 

\begin{proposition}
    Under Assumptions \ref{assump:symmetric_markov_bassa}, \ref{assump:pi}, \ref{assump:kappa}, the process $Y$ killed at rate $\kappa$, that is, with lifetime given by $\eqref{eq:taud}$, defines a strongly continuous sub-Markovian semigroup ${\{\exp(-tL^\kappa):t \ge 0\}}$ on $C_0(\Rd)$ with symmetric, continuous kernel $p^\kappa(t,x,y)$. The corresponding generator $-L^\kappa = -L^0 \dot- \kappa$, extends $-L^0-\kappa$. In addition, it has Feynman--Kac representation,
\begin{equation*}
\begin{split}
               \left[\exp\left(-tL^\kappa\right)f\right](x)&= \int p^\kappa (t,x,y) f(y) \dif m(y)\\
               &=\E_x\left[\exp\left(-\int_0^t \kappa(Y_s) \dif s\right) f(Y_t) \right].
\end{split}
\end{equation*}

    Furthermore, the semigroup is strongly continuous on $\L^p(\Gamma)$ for any $1 \le p <\infty$. In particular, on $\L^2(\Gamma)$, it is self-adjoint and possesses a self-adjoint generator. 
    \label{prop:generator}
\end{proposition}
\begin{proof}
    See Appendix~\ref{sec:proof_generator}.
\end{proof}

As before, when we want to make explicit which $\L^p(\Gamma)$ space we are using, for $1\le p < \infty$, we will write $-L_p^\kappa$ for the generator of the strongly continuous semigroup on $\L^p(\Gamma)$, with corresponding domain $\Dom(L_p^\kappa)\subset \L^p(\Gamma)$. The domain of the generator
may be defined as the image of the semigroup acting on $\L^p(\Gamma)$.

\begin{remark}
    It follows from Assumption~\ref{assump:pi} that $\pi \in \Dom(L^\kappa_2)$, since both $\pi$ and $\mu$ are in $\L^2(\Gamma)$, and formally $L^\kappa \pi = C\mu$.
\end{remark}

We have one final technical assumption.
\begin{assumption}[Technical conditions on $\pi, \mu$]
    We have that 
    \begin{equation}
     \pi \in \Dom(L_1^\kappa)   ,\quad L_1^\kappa \pi = C\mu.
        \label{eq:tech_assump_pi_L1}
    \end{equation}
    Furthermore, $\mu$ is such that
    \begin{equation}
        \int \dif \Gamma(x) \mu(x)\, \E_x^0\left[ \sup_{t \in [0,1]} \left|\kappa(Y_t) \mathrm e^{-\int_0^t \kappa(Y_s) \dif s}\right| \right] <\infty.
        \label{eq:tech_cond_mu}
    \end{equation}
    \label{assump:technical}
\end{assumption}

The condition \eqref{eq:tech_assump_pi_L1} is fairly abstract, and so might be
difficult to verify in a particular case, or in a general class of
processes that one may want to consider.
Lemma~\ref{lemma:L1_pi_cond} gives a sufficient condition which we will make use of. Set
\begin{equation*}
    \bar \pi := \pi \gamma.
\end{equation*}
We write $W^{2,1}(\R^d)$ for the Sobolev space of measurable functions on $\R^d$ whose first and second derivatives are integrable with respect to Lebesgue measure on $\R^d$. 

\begin{lemma}
    Assume that Assumptions~\ref{assump:symmetric_markov_bassa}, \ref{assump:pi}, \ref{assump:kappa} hold. Suppose that the drift is at most linear in the tails: we can bound $|\nabla A(x)| \le K |x|$, for some $K>0$, for all $x$ outside of some compact set. Suppose $\pi$ is smooth, and that $\bar \pi \in W^{2,1}(\R^d)$. In addition, we require that
    \begin{equation*}
         \int_{\R^d}|\nabla A(x) \cdot \nabla \bar\pi(x)|\dif x <\infty,\quad \int_{\Rd} |\Delta A(x) \bar\pi(x)|\dif x <\infty.
    \end{equation*}
    Then \eqref{eq:tech_assump_pi_L1} holds.
    \label{lemma:L1_pi_cond}
\end{lemma}
\begin{proof}
    See Appendix~\ref{sec:proof_L1_pi_cond}.
\end{proof}

Alternatively, \eqref{eq:tech_assump_pi_L1} will automatically
hold whenever $\Gamma$ is a finite measure. This is the case whenever the underlying diffusion $Y$ is positive recurrent, say a stable Ornstein--Uhlenbeck process. Then under $\pi \in \L^2(\Gamma)$ and $\mu \in \L^2(\Gamma)$, $\pi \in \Dom(L_1^\kappa)$ with $L_1^\kappa \pi = C \mu$, 
since in that case $\L^2$ convergence implies $\L^1$ convergence.

The condition \eqref{eq:tech_cond_mu} is needed so that we can differentiate under the integral. A necessary condition for \eqref{eq:tech_cond_mu} to hold is that $\int \dif \Gamma(x)\mu(x)\kappa(x)<\infty$, so in particular $\mu$ cannot have tails which are too heavy relative to $\pi$. From a computational point of view, this is reasonable since otherwise the regeneration mechanism would be highly inefficient; the Restore process would tend to regenerate very rapidly. 
Of course, a sufficient condition for \eqref{eq:tech_cond_mu} is that
    \begin{equation*}
        \int \dif \Gamma(x) \mu(x)\, \E_x^0\left[ \sup_{t \in [0,1]} \kappa(Y_t)  \right] <\infty.
    \end{equation*}

\begin{theorem} \label{thm:restore_invariant_diffusion}
    Under Assumptions \ref{assump:symmetric_markov_bassa}, \ref{assump:pi}, \ref{assump:kappa}, \ref{assump:technical}, the Restore process $X$ with interarrival dynamics $Y$, regeneration rate $\kappa$ and regeneration density $\mu$ has invariant distribution $\pi$.
    \label{thm:pi_inv_symm}
\end{theorem}
\begin{proof}
    See Appendix~\ref{sec:proof_restore_invariant_diffusion}.
\end{proof}

\subsubsection{Examples}
\label{subsubsec:BASSA_ex}
We now give some examples of diffusions which satisfy the assumptions of Theorem \ref{thm:pi_inv_symm}.

Sufficient conditions ensuring BASSA are given in Example 2 of \cite[Chapter 1.C]{Demuth2000}. In our present setting when we consider diffusions defined by \eqref{eq:SDE}, these conditions can be written as
\begin{align}
    \label{eq:BASSA_cond1}
    \exp(A(x)) \ge c^{-1}\exp(-c|x|^2), \quad &\forall x \in \Rd,\\
    \label{eq:BASSA_cond2}
    c^{-1} \le \exp(A(x)-A(y)) \le c,\quad &\forall x,y \in \Rd: |x-y| \le c^{-1} (1+|x|)^{-c},
\end{align}
for some $c>0$.

Let $|\cdot|$ denote the $\ell_2$ norm on $\Rd$.
\begin{proposition}
    The SDE \eqref{eq:SDE} with $A = \alpha |x|^2$ for any $\alpha \in \R$ satisfies BASSA.
\end{proposition}
\begin{remark}
    In this case $\nabla A(x) = 2\alpha x $ is linear. $\alpha <0$ corresponds to a (stable) Ornstein--Uhlenbeck process, $\alpha =0$ is a Brownian motion and $\alpha >0$ is an \textit{unstable} Ornstein--Uhlenbeck process which drifts into the tails.
\end{remark}
\begin{proof}
    \eqref{eq:BASSA_cond1} clearly holds in this setting. The second condition \eqref{eq:BASSA_cond2} can be seen from the reverse triangle inequality:
\begin{align*}
    \left| |x|^2 - |y|^2\right | &= (|x|+|y|)\left||x|-|y|\right | \le (|x|+|y|)|x-y|\\
    &\le \frac{|x|+|y|}{c(1+|x|)^c} \le \frac{2|x|}{c(1+|x|)^c} +\frac{1}{c^2 (1+|x|)^{2c}}.
\end{align*}
This is uniformly bounded over $x\in\Rd$ for $c>1$.
\end{proof}

\subsection{Jump processes}
\label{subsec:jump_proc}
The Restore process is inherently a continuous-time process, and so the underlying process $Y$ must be a continuous-time object. Suppose, however, we are given a a discrete-time Markov transition kernel $P$ on $(E,m)$, with action on measurable functions $f:E\to \R$ and measures $\nu$ on $E$ given by
\begin{equation}
    \begin{split}
        Pf(x) &= \int f(y) p(x,y) \dif m(y),\quad x \in E, \\
        \nu P(\dif y) &= \int \nu(\dif x) \, p(x,y) \dif m(y),
    \end{split}
    \label{eq:trans_kernel}
\end{equation}
for some integral kernel $p(x,y)$ on $E\times E$, whenever these integrals make sense. Since we have an integral kernel $p(x,y)$, we will also think of $\nu P$ as a measurable function given by
\begin{equation*}
    \nu P(y) := \int \nu(\dif x) \,p(x,y)
\end{equation*}
for a measure $\nu$ on $E$, provided this makes sense.

It is straightforward to embed $P$ into continuous time, by specifying a measurable function $\lambda:E \to \R^+$, the \textit{holding rates}.
We take the jump chain to be defined by the discrete-time Markov kernel $P$, and just take the holding times to be independent Exp$(\lambda(x))$ times, when currently at state $x$. 

Such a process will be a \textit{continuous-time jump process} on $E$, meaning it has right-continuous, piecewise-constant sample paths. Provided they are nonexplosive, such processes are determined by the transition kernel of the jump chain and the holding rates. See for instance, \cite[Chapter 4.2]{Ethier1986}.

Suppose $\pi, \mu$ are two densities on $E$ with respect to $m$, the target density and regeneration density respectively, where we assume $\pi$ is positive. Suppose we are given a transition kernel $P$ on $E$ and holding rates $\lambda: E \to \R^+$. We now construct the Restore process. Given a constant $C$, define the regeneration rate $\kappa$ to be
\begin{equation}
    \kappa(x) := \frac{ \int \pi(y)\lambda(y) p(y,x)m(\dif y) - \lambda(x) \pi(x)  }{\pi(x)}+C\frac{ \mu(x)}{\pi(x)} , \quad x \in E.
    \label{eq:kappa_jump_process}
\end{equation}
\begin{assumption}[Jump process Restore]
    $P$ is a transition kernel with a density as in \eqref{eq:trans_kernel}, $\lambda: E\to \R^+$ is measurable, strictly positive. $\pi$ is a positive probability density with respect to $m$, $\mu$ is a probability density with respect to $m$, and $\int \lambda(x)\pi(x)m(\dif x)<\infty$. The constant $C$ is such that $\kappa\ge 0$ on $E$. $\kappa$ is locally bounded, and we have that $\int (\lambda(x)+\kappa(x))^2\pi(x) m(\dif x)<\infty$.
    \label{assump:MCMC}
\end{assumption}
 Note that such jump processes are right processes (Exercise 14.18 of \cite{Sharpe1988}).

\begin{remark}
    From \eqref{eq:kappa_jump_process}, we see that a sufficient condition for $\kappa(x)\ge 0$ is that
    \begin{equation*}
        C\mu(x) \ge \lambda(x) \pi(x), \quad x \in E.
    \end{equation*}
    Alternatively, if the underlying process is already $\pi$-invariant, so $\pi Q^0 \equiv 0$, then the first term in \eqref{eq:kappa_jump_process} is identically zero and any $C>0$ and $\mu$ may be chosen.
\end{remark}
\begin{remark}
    When the state space $E$ is countable, the Markov process $Y$ is necessarily a jump process, with law defined by a transition rate matrix $Q$. Probability distributions on $E$ are given by (possibly infinite) row vectors. In this case, the regeneration rate \eqref{eq:kappa_jump_process} can be written simply as
    \begin{equation*}
        \kappa(x) := \frac{(\pi Q)(x)}{\pi(x)} + C \frac{\mu(x)}{\pi(x)}, \quad x \in E,
    \end{equation*}
    where $\pi Q$ is a row vector obtained by straightforward matrix multiplication.
\end{remark}

This construction can also be extended to kernels $P$ which do not possess a density as in \eqref{eq:trans_kernel}. For example, the classical Metropolis--Hastings kernel is of the form
\begin{equation*}
    P(x,\dif y) = \alpha(x,y)q(x,y)\dif m( y) + (1- j(x)) \delta_x(\dif y),
\end{equation*}
where $0\le \alpha(x,y)\le 1$ are the acceptance probabilities, $q(x,y)$ is a transition density (so $\int q(x,y)\dif m (y) = 1$ for each $x\in E$), and 
$$j(x):=\int \alpha(x,y)q(x,y)\dif m( y)$$ 
are the jump probabilities. Because of the presence of the delta mass $\delta_x(\dif y)$, such kernels cannot possess straightforward densities. However in continuous-time, these rejected moves associated with the delta mass are not visible, and so we can modify the regeneration rate as follows: we replace the term $(\pi\lambda) P(x)$ in \eqref{eq:kappa_jump_process} by
\begin{equation*}
    \int \dif m(y)\, \pi(y)\lambda(y) \alpha(y,x) q(y,x)  + \lambda(x)(1- j(x)) \pi(x).
\end{equation*} 

Returning to the construction of the Restore process, we will take the interarrival dynamics to be given by the jump process defined by $P$ and $\lambda$, the regeneration rate to be $\kappa$ and the regeneration density $\mu$. The resulting Restore process $X$ is another continuous-time jump process, and so we describe its jump chain and holding rates. This will provide a method to simulate the process.

At $x\in E$, the transition kernel $P^\mu(x,\dif y)$ of the jump chain is given by
\begin{equation*}
    P^\mu(x,\dif y) = \frac{\lambda(x)}{\lambda(x)+\kappa(x)} {P(x,\dif y)} + \frac{\kappa(x)}{\lambda(x)+\kappa(x)}\mu(y)\dif m(y).
\end{equation*}
The overall holding rates in continuous time are given by
\begin{equation*}
    \bar\lambda(x) = \lambda(x)+ \kappa(x),  \quad x \in E,
\end{equation*}
that is, at $x \in E$, by the Markov property, the time until the next jump is an Exp$(\bar\lambda(x))$ time.

\begin{theorem} \label{thm:jump_invar}
    Assume that Assumption \ref{assump:MCMC} holds, and that the interarrival dynamics defined by $P$ and $\lambda$ are nonexplosive. Then the resulting Restore process $X$ is a nonexplosive jump process with invariant distribution $\pi$.
    \label{thm:jump_restore}
\end{theorem}
\begin{proof}
    Nonexplosivity follows from Lemma~\ref{prop:Restore_exist}, and the fact that the interarrival process is assumed nonexplosive.

    Let us write $\{Q^\mu_t: t\ge 0\}$ for the continuous-time semigroup for the Restore process $X$. Our goal is to show that $\pi Q^\mu_t f = \pi[f]$ for any continuous bounded function ${f: E\to \R}$, for each $t \ge 0$. To do this we compute the time derivative of the mapping $t \mapsto \pi Q^\mu_t f$, and show that it is 0. By time-homogeneity and the semigroup property, it is sufficient to compute this derivative at $t=0$.
    This was the approach similarly used to prove $\pi$-invariance of the Bouncy Particle Sampler in the supplementary material of \cite{Bouchard-Cote2018}.
    
    By conditioning on the first jump, we obtain the following representation (\textit{cf}. equation (4.24) of \cite{Moyal1957}),
    \begin{equation*}
        \begin{split}
                    Q_t f(x) &= \mathrm e^{-\bar\lambda(x)t} f(x) + \int_0^t \dif s \, \bar\lambda(x)\mathrm e^{-\bar\lambda(x)s} P^\mu[ Q_{t-s} f](x)\\
                    &= \mathrm e^{-\bar\lambda(x)t} f(x) + \int_0^t \dif s \,\bar\lambda(x) \mathrm e^{-\bar\lambda(x) (t-s)} P^\mu [Q_s f](x).
        \end{split}
    \end{equation*}
    From this representation we can calculate the derivative,
    \begin{multline*}
                \frac{\dif Q_t f(x)}{\dif t} = -\bar\lambda(x) \mathrm e^{-\bar\lambda(x)t} f(x) + \bar\lambda(x) P^\mu[Q_t f](x) \\
            - \int_0^t \dif s \, \bar\lambda(x)^2 \mathrm e^{-\bar\lambda(x) (t-s)}P^\mu [Q_s f](x).
    \end{multline*}
    At $t=0$ the definitions of $\bar\lambda$ and $\kappa$ imply that 
    $\frac{\dif}{\dif t} \pi Q_t f = 0$. The exchange of integration and differentiation is justified by the assumption ${\pi(\bar\lambda^2) <\infty}$.
\end{proof}

In this setting, practical simulation of the Restore process is straightforward, \textit{even when the regeneration rate is unbounded}, since the interarrival process $Y$ is piecewise-constant. See Algorithm~\ref{alg:restore_jump} for one possible implementation.

\begin{algorithm}
\caption{Jump process Restore Sampler.}
\begin{algorithmic}[1]
\State \textit{initialize:} $X_0=x_0, t_0=0, i=0$
\While{$t_i<T$}
\State $i \gets i+1$
\State simulate $\tau_{i-1}^{(1)} \sim $ Exp$(\lambda(X_{i-1}))$, $\tau_{i-1}^{(2)} \sim $ Exp$(\kappa(X_{i-1}))$
\State $\tau_{i-1} \gets \tau_{i-1}^{(1)} \wedge \tau_{i-1}^{(2)}$
\State $t_i \gets t_{i-1}+\tau_{i-1}$
\If{$\tau_{i-1}^{(1)} <\tau_{i-1}^{(2)}$}
\State $X_i \sim P(X_{i-1}, \cdot)$
\Else 
\State $X_i \sim \mu$
\EndIf
\EndWhile
\State \textbf{end while}
\State \textbf{return} pairs $(t_i, X_i)$
\end{algorithmic}
\label{alg:restore_jump}
\end{algorithm}
Algorithm~\ref{alg:restore_jump} can be seen as a continuous-time variant of standard Metropolis--Hastings; at each iteration we `propose' a move according to $P(X_{i-1},\cdot)$, which is either accepted or rejected, depending on two exponential clocks. Upon rejecting a move, rather than remaining at $X_{i-1}$ instead we move to a new location drawn from $\mu$.

\section{Limiting properties}
\label{sec:limiting_properties}
In this section we consider some limiting properties of the Restore process. We will not \textit{a priori} assume that $X$ has invariant distribution $\pi$, but will work in the abstract framework of Lemma~\ref{prop:Restore_exist}: The underlying process $Y$ is a right process evolving on a Radon space $(E,m)$, we have a locally bounded measurable function $\kappa:E\to \R^+$, and we a probability measure $\mu$ on $E$. We consider the Restore process $X$ with these dynamics. We will write ${\{P_t^\mu: t\ge 0\}}$ for its semigroup.

\subsection{Central Limit Theorem}
\label{subsec:CLT}
We first give a central limit theorem for the Restore process. Our approach here is inspired by \cite{Hobert2002}, who considered regenerative methods for MCMC (in discrete time).

We fix a measurable function $f:E\to \R$.
\begin{assumption}[Central limit theorem]
    We assume the basic conditions of Lemma~\ref{prop:Restore_exist}. Furthermore we assume that $X$ is irreducible, 
    \begin{equation}
        \E_\mu[\taud^2] <\infty,
        \label{eq:CLT_moment}
    \end{equation}
    and that our function $f:E\to \R$ satisfies
    \begin{equation*}
        \E_\mu\left [\left(\int_0^{\taud} f(X_s)\dif s\right)^2 \right] <\infty.
    \end{equation*}
    \label{assump:CLT}
\end{assumption}
A sufficient condition for Assumption \ref{assump:CLT} to hold is that $f$ is a bounded function and we have simply the second moment condition \eqref{eq:CLT_moment}. In turn, a sufficient condition for \eqref{eq:CLT_moment} is that Assumption \ref{assump:CFTP} holds, since in that case $\taud$ can be stochastically dominated by an Exp($\ubar\kappa$) random variable.

Under Assumption \ref{assump:CLT} we will see that a central limit theorem holds. This can be easily done since the lifetimes of the Restore process, by construction, are independent and identically distributed.

As in the construction of Restore in Section \ref{sec:construction}, set $T_0=0$, let $(T_n)$ be the successive regeneration times and let $(\tau^{(i)})$ be the lifetimes. We take the initial distribution $X_0\sim \mu$. Set for each $i=0,1,2,\dots$,
\begin{equation*} 
    Z_i := \int_{T_{i}}^{T_{i+1}} f(X_s)\dif s.
\end{equation*}
By construction the $(Z_i)$ are independent and identically distributed, with finite first and second moments. 

We can apply the strong law of large numbers to the following numerator and denominator:
\begin{equation*}
    \frac{\int_0^{T_n}f(X_s)\dif s }{T_n}=\frac{\sum_{i=0}^{n-1} Z_i}{\sum_{i=0}^{n-1} \tau^{(i)}} \to \frac{ \E_\mu \left [ \int_0^{\tau^{(0)}} f(X_s) \dif s \right ] }{\E_\mu[\tau^{(0)}]}
\end{equation*}
 almost surely as $n\to\infty$.
 
Let us write
 \begin{equation*}
     \pi[f] := \frac{ \E_\mu \left [ \int_0^{\tau^{(0)}} f(X_s) \dif s \right ] }{\E_\mu[\tau^{(0)}]}.
 \end{equation*}
When the process is ergodic, this corresponds to the invariant distribution of the Restore process. It follows immediately that the random variables
\begin{equation*}
    Z_i - \tau^{(i)} \pi[f], \quad i=0,1,2,\dots
\end{equation*}
are independent and identically distributed and have mean 0 under $\E_\mu$. 

Now we set, in analogue with the expression given in \cite{Hobert2002},
\begin{equation}
    \sigma^2_f := \frac{\E_\mu\left[ \left (Z_0 - \tau^{(0)} \nu_\pi[f]\right)^2 \right]}{\left(\E_\mu[\tau^{(0)}]\right)^2}.
    \label{eq:rest:var_f}
\end{equation}
This numerator is finite by Assumption \ref{assump:CLT}.
 
 \begin{theorem}[Central limit theorem]
     We have that 
     \begin{equation}
         \sqrt n \left ( \frac{\int_0^{T_n} f(X_s)\dif s}{T_n} - \nu_\pi[f] \right) \overset{d}\to N\left (0, \sigma_f^2\right).
         \label{eq:CLT}
     \end{equation}
     \label{thm:CLT}
 \end{theorem}
 
\begin{proof}
    The left-hand side of \eqref{eq:CLT} can be written
    \begin{equation*}
        \sqrt{n}\left(\frac{\sum_{i=0}^{n-1} Z_i}{\sum_{i=0}^{n-1} \tau^{(i)}} - \nu_\pi[f]\right)
        =\frac{1}{\sqrt n}\cdot \frac{n}{\sum_{i=0}^{n-1} \tau^{(i)}} \left( \sum_{i=0}^{n-1} \left( Z_i - \tau^{(i)} \nu_\pi[f]\right) \right).
    \end{equation*}
    By the strong law of large numbers and the continuous mapping theorem, $n/\sum_{i=0}^{n-1} \tau^{(i)}$ converges almost surely to $(\E_\mu[\tau^{(0)}])^{-1}$, and in distribution also.
    
    Hence by applying Slutsky's lemma and the central limit theorem to the independent and identically distributed mean zero random variables $(Z_i - \tau^{(i)})$, we see that \eqref{eq:CLT} holds.
\end{proof}

Let us write $\bar \tau_n := n^{-1}\sum_{i=0}^{n-1} \tau^{(i)}$ and $\bar f_n := \frac{\int_0^{T_n}f(X_s)\dif s }{T_n}$. Similar to \cite{Hobert2002}, our $\sigma^2_f$ can be consistently estimated by
\begin{equation*}
    \hat \sigma^2_f := \frac{\sum_{i=0}^{n-1} \left( Z_i - \bar f_n \tau^{(i)} \right)^2}{n \bar\tau_n^2}.
\end{equation*}
This is because the difference between $\hat\sigma^2_f$ and
\begin{equation*}
    \frac{\sum_{i=0}^{n-1} \left( Z_i - \tau^{(i)} \nu_\pi[f] \right)^2}{n \bar \tau_n^2}
\end{equation*}
converges to zero almost surely as $n\to \infty$, and the latter is a consistent estimator for $\sigma^2_f$.

We can use this to get an estimate of the efficiency of Restore. If we let
\begin{equation*}
    v_\pi(f) := \int \left (f(x) - \nu_\pi[f]\right)^2 \dif \pi(x),
\end{equation*}
then we can set the effective sample size $n_{\text{eff}}$ to be
\begin{equation*}
    n_{\text{eff}} := \frac{v_\pi(f)}{\sigma^2_f},
\end{equation*}
which we may be able to estimate.

We see from \eqref{eq:rest:var_f}, that the denominator $(\E_\mu[\tau^{(0)}])^2$ will have a significant influence on the overall variance. If $\E_\mu[\tau^{(0)}]$ is small, the resulting variances of individual lifetimes may be unacceptably large, and as such practically speaking it is important to choose the regeneration distribution in such a way that the lifetimes are (on average) not too short. In particular, this means choosing $\mu$ which avoids regions where the regeneration rate is particularly high.

\subsection{Coupling from the past}
\label{subsec:CFTP}

Under additional (fairly strong) conditions, we will have direct access to the stationary distribution of the Restore process.
\begin{assumption}[Coupling from the past]
    There exists some $\ubar\kappa >0$ such that $m$-almost everywhere,
    \begin{equation*}
        \kappa \ge  \ubar\kappa >0.
    \end{equation*}
    \label{assump:CFTP}
\end{assumption}

We write $\|\cdot\|_\infty$ for the sup norm of a bounded function and $\|\cdot\|_1$ for the total variation norm signed measures; given a signed measure $\nu$, $$\|\nu\|_1 = \sup\{|\nu(f)|: f \text{ bounded, measurable }, \|f\|_\infty \le 1\}.$$
\begin{proposition}[Uniform ergodicity]
    Assume the basic conditions of Lemma~\ref{prop:Restore_exist} hold, and that $X$ is irreducible. Under Assumption \ref{assump:CFTP}, the Restore process $X$ is uniformly geometrically ergodic, meaning that there exists a unique invariant distribution $\pi$ such that
    \begin{equation*}
        \|\nu P^\mu_t - \pi\|_1 \le 2 \mathrm e^{-t \ubar\kappa},
    \end{equation*}
    for any initial distribution $\nu$ and $t\ge 0$.
    \label{prop:unif_erg}
\end{proposition}
\begin{proof}
    Fix any two arbitrary initial distributions $\nu_1, \nu_2$ on $E$. By Assumption~\ref{assump:CFTP} and Poisson superposition, we can decompose the Poisson process of regeneration times as the superposition of two independent Poisson processes: a \textit{homogeneous} Poisson process $N_1$ of rate $\ubar \kappa$, and an inhomogeneous Poisson process $N_2$ with rate function $t\mapsto \kappa(X_t)-\ubar\kappa$. Thus we can couple two copies of the Restore process $X$, with initial distributions $\nu_1$ and $\nu_2$ respectively, by constructing them to have $N_1$ in common, and the same regeneration locations. The two processes will then meet at the first arrival time of $N_1$ and evolve identically thereafter. 
    
    Hence by the well-known coupling inequality (see, for instance, \cite[Section 1.5.4]{Thorisson2000}),
    \begin{equation}
        \|\nu_1 P^\mu_t - \nu_2 P^\mu_t\|_1 \le 2 \mathrm e^{-t \ubar\kappa}.
        \label{eq:unif_erg}
    \end{equation}
    The Markov property (i.e. the semigroup property) then shows that for any initial distribution $\nu$, $(\nu P_t^\mu)_{t\ge 0}$ forms a Cauchy sequence in the space of probability measures equipped with the total variation norm. By completeness, there exists a limiting probability distribution $\pi$, which must also be a stationary distribution, by the Markov property and the fact that $P_t^\mu$ is a contraction in $\| \cdot\|_1$. That is, we have $\pi P_t^\mu = \pi$ for any $t \ge 0$. By irreducibility, this invariant distribution is unique. Thus taking $\nu_1 = \nu$ and $\nu_2 = \pi$ in \eqref{eq:unif_erg} the Proposition is proven.
\end{proof}

In fact under Assumption \ref{assump:CFTP} we can do even better than uniform ergodicity and employ \textit{coupling from the past} (CFTP), a technique pioneered by \cite{Propp1996} to obtain exact draws from the stationary distribution $\pi$.
For a related approach to exact MCMC methods, see the recent approach of \cite{Jacob2020} using couplings.


\begin{theorem}[Coupling from the past]
    Under the conditions of Lemma~\ref{prop:Restore_exist} and Assumption \ref{assump:CFTP}, consider the Restore process $X$ with interarrival dynamics $Y$, modified regeneration rate $$\kappa':= \kappa - \ubar\kappa \ge 0,$$ and regeneration density $\mu$. Suppose $X$ is irreducible and has initial distribution $$X_0 \sim \mu.$$ Let $T \sim \text{Exp}(\ubar \kappa)$ be independent of $X$. Then 
    \begin{equation*}
        X_T \sim \pi,
    \end{equation*}
    \label{thm:CFTP}
    where $\pi$ is the unique invariant distribution of the process.
\end{theorem}
\begin{proof}
    This follows from the technique of \cite{Propp1996}. We saw in the proof of Proposition~\ref{prop:unif_erg} that we can realise the Poisson process of regeneration times as the superposition of two independent Poisson processes: a homogeneous Poisson process $N_1$ of rate $\ubar \kappa$ and an inhomogeneous Poisson process $N_2$ with rate $t\mapsto \kappa'(X_t)$. As in \cite{Propp1996}, we imagine a Restore process $X$, initialised from some arbitrary initial distribution at time $-\infty$, run until time 0. Since we have established uniform ergodicity in Proposition~\ref{prop:unif_erg}, we know that $X_0\sim \pi$. Let $-T$ be the most recent arrival of $N_1$ before time 0. Regardless of the prior evolution of $X$, we know that $X_{-T}\sim \mu$ as $-T$ was a regeneration time. Since $N_1$ and $N_2$ are independent, $X_0$ then has the same law as a Restore process at time $T$, initialised from $\mu$, with regeneration rate $\kappa'$.
    
    Since the time reverse of a homogeneous Poisson process is also a homogeneous Poisson process, we can instead imagine initialising $X_0 \sim \mu$ and evolving an exponential time $T$ into the future with modified regeneration rate $\kappa'$.
\end{proof}

In the case when $\kappa$ is bounded above, one implementation is given in Algorithm~\ref{alg:bdd_CFTP} below. In this case, simulation of the lifetimes $\tau_\partial$ is straightforward, since can make use of Poisson thinning; see, for instance, \cite[Chapter 6.2]{Devroye1986}. 
\begin{algorithm}
\caption{Bounded Restore Sampler: $\kappa \le M$, with CFTP.}
\begin{algorithmic}[1]
\State \textit{draw run time:} $T\sim \text{Exp}(\ubar \kappa)$
\State \textit{initialize:} $X_0\sim \mu, t_0=0, i=0$
\State $i\gets i+1$
\State $t_i \gets t_{i-1}+\tau_{i-1}$, where $\tau_{i-1} \sim \text{Exp}(M-\ubar \kappa)$
\While{$t_i<T$}
\State simulate $Z_i \sim \mathcal L(Y_{\tau_{i-1}}|Y_0=X_{i-1})$
\State \textbf{with probability} $1-(\kappa(X_i)-\ubar \kappa)/(M-\ubar \kappa)$
\State $\quad$  $X_i \gets Z_i$
\State \textbf{else}
\State $\quad$  $X_i \sim \mu$
\State $i \gets i+1$
\State $t_i \gets t_{i-1}+\tau_{i-1}$, where $\tau_{i-1} \sim \text{Exp}(M-\ubar\kappa)$
\EndWhile
\State \textbf{end while}
\State simulate $Z \sim \mathcal L(Y_{T-t_{i-1}}|Y_0=X_{i-1}) $
\State \Return $Z$, which is drawn exactly from $\pi$
\end{algorithmic}
\label{alg:bdd_CFTP}
\end{algorithm}

This CFTP implementation can be seen as a continuous-time version of the multigamma coupler of \cite{Murdoch1998} or of the hybrid scheme of \cite[Section 3]{Murdoch2000}. 
The multigamma coupler of \cite{Murdoch1998} assumes we have a discrete-time Markov chain whose transition kernel $P$ satisfies $P(x,\dif y)=f(y|x)\dif y$, where $f(y|x)\ge r(y)$, for all $x$, for some nonnegative function $r$ which satisfies $\rho :=\int r(y)\dif y >0$. Let $\nu_r$ denote the probability distribution with density (proportional to) $r$.
Thus when simulating the chain, at each step with probability $\rho$, the chain will move to a point drawn from $\nu_r$, independent of the current location. This enables a CFTP construction, the multigamma coupler; see \cite[Section 2.1]{Murdoch1998}.

This uniform probability $\rho$ is precisely what enables CFTP to be applied. It informally says that independent of location, at each discrete time step all locations are trying to couple with probability $\rho$ to the same point, drawn from $\nu_r$. This plays the same role as our homogeneous rate $\ubar \kappa$, which informally states that in continuous time, at \textit{rate} $\ubar \kappa$, all locations are trying to couple to the same location, drawn from $\nu_\mu$.

A crucial difference between our approaches, however, is that our underlying dynamics $Y$ do not themselves have to be $\pi$-invariant; in fact we will see in Section~\ref{sec:Restore_examples} an example where the local process does not possess an invariant distribution at all.

\subsubsection{Example: Classical rejection sampler}
We show that the classical rejection sampler can be seen as a special case of the CFTP implementation of the Restore process. A similar result was established for the Independence Sampler in \cite{Murdoch1998}.

Let $\pi, \mu$ be density functions on $E$ with respect to $m$. We take $Y$ to be the trivial stochastic process on $E$ which given its initial position $Y_0$, has constant sample paths: almost surely, $Y_t = Y_0$ for all $t \ge 0$. Define the regeneration rate
\begin{equation}
    \kappa(x) = C \frac{\mu(x)}{\pi(x)}, \quad x \in E,
    \label{eq:kappa_rej}
\end{equation}
for any constant $C>0$. If we were to implement the classical rejection sampler targeting $\pi$ from $\mu$ we would require the following condition:
\begin{equation}
    \pi(x) \le M \mu(x), \quad x \in E,
    \label{rej_samp_assump}
\end{equation}
for some (finite) constant $M$. The classical rejection sampler targeting $\pi$ from $\mu$ repeatedly draws $X_n$ independently from $\mu$, and accepts it with probability $\pi(X_n)/(M\mu(X_n))$, otherwise rejects it and tries again with a new $X_{n+1}\sim \mu$. The final accepted value $X_n$ is an exact draw from $\pi$.

\begin{theorem}
    Under \eqref{rej_samp_assump}, the CFTP implementation of the Restore process (Theorem \ref{thm:CFTP}) with constant interarrival dynamics, regeneration rate $\kappa$ as in \eqref{eq:kappa_rej} and regeneration density $\mu$ is identical to classical rejection sampling targeting $\pi$ from $\mu$.
    \label{thm:rej_samp}
\end{theorem}
\begin{proof}
    We see that \eqref{rej_samp_assump} holds if and only if Assumption \ref{assump:CFTP} holds with $$\ubar \kappa = C/M.$$
    
    Under this condition in the CFTP implementation (Theorem \ref{thm:CFTP}) we run the Restore process with regeneration rate 
    \begin{equation*}
        \kappa' = \kappa - \ubar\kappa = C \frac{\mu}{\pi} - \frac{C}{M}
    \end{equation*}
    for a time $T\sim \text{Exp}(C/M)$.
    
    We can simulate this Restore process iteratively by drawing for each $n$, $X_n \sim \mu$. We have two competing independent exponential clocks, $T\sim \text{Exp}(C/M)$ and $T_n \sim \text{Exp}(\kappa'(X_n))$.
    
    If $T < T_n$, all trajectories have coupled and so we terminate the algorithm and output $X_n$, which is an exact draw from $\pi$. By the theory of competing exponentials this occurs with probability
    \begin{equation*}
        \frac{C/M}{C/M + C \left(\frac{\mu(X_n)}{\pi(X_n)}-\frac{1}{M} \right)}= \frac{\pi(X_n)}{M\mu(X_n)}.
    \end{equation*}
    This is exactly the probability of acceptance for the classic rejection sampler.
    
    If $T \ge T_n$ then we iterate again and draw $X_{n+1}\sim \mu$, $T_{n+1}\sim \text{Exp}(\kappa'(X_{n+1}))$. By the memoryless property of the exponential distribution we have again two independent exponential clocks as before.
\end{proof}

If \eqref{rej_samp_assump} doesn't hold, provided there is a unique invariant distribution $\pi$ we can still use ergodic averages to estimate $\nu_\pi[f]$ for any bounded $f$. Suppose we run the Restore process with constant interarrival dynamics, regeneration rate $\kappa$ as in \eqref{eq:kappa_rej} and regeneration density $\mu$ for $n$ complete lifetimes. The corresponding ergodic average is
\begin{equation*}
    \frac{1}{T_n}\sum_{i=1}^n f(X_i) \tau^{(i)},
\end{equation*}
where $X_i \sim \mu$ are i.i.d., conditional on $X_i$, $\tau^{(i)} \sim \text{Exp}(C \mu(X_i)/\pi(X_i))$ are independent and $T_n = \sum_{i=1}^n \tau^{(i)}$.
Thus the estimator of $\pi[f]$ can be seen as an importance sampling--type estimator with randomized importance weights; note $C\E[\tau^{(i)}|X_i] = \pi(X_i)/\mu(X_i)$.

\section{Practical considerations}
\label{sec:simulation} \label{sec:implementation}

We consider now some practical questions related to the Restore process.

\subsection{Minimal regeneration distribution}
\label{subsec:minimal_mu}
In this section we assume that we are given some fixed interarrival process, a positive target density $\pi$ on $E$ and a regeneration density $\mu$ on $E$, which are both normalized.

The most significant challenge for implementing the Restore sampler is 
to ensure that the regeneration rate is nonnegative;
we need to find a constant $C$ so that
\begin{equation}
    \kappa(x) = \tilde \kappa(x) + C \frac{\mu(x)}{\pi(x)}\ge \ubar\kappa \ge 0 \text{ for all } x \in E,
    \label{eq:kap_nonneg}
\end{equation}
for some nonnegative constant $\ubar\kappa$. 
Here $\tilde \kappa$ is defined in \eqref{eq:kappa_tilde} for the diffusion setting and for the jump process setting is defined to be the first term on the right-hand side of \eqref{eq:kappa_jump_process}.
As shown in the proof of Theorem~\ref{thm:pi_inv_symm}, $C= \E_\mu[\taud]$ can be interpreted as the average lifetime when started from $\mu$.


One natural way to choose the regeneration density $\mu$ and constant $C$ is to minimize the
number of regeneration events.
That is, we would like to choose some \textit{minimal} regeneration distribution $\mu^*$ and constant $C^*$\ such that the regeneration rate is given by
\begin{equation}
    \kappa^* := \tilde \kappa + C^*\, \frac{\mu^*}{\pi} = \tilde \kappa \vee \ubar \kappa.
    \label{eq:opt_kappa}
\end{equation}
This is entirely analogous to the choice of bounce rate for the Bouncy Particle Sampler of \cite{Bouchard-Cote2018},
and of the canonical switching rate for the Zig-Zag in \cite{Bierkens2019}.
In order to satisfy \eqref{eq:opt_kappa}, the appropriate choice of density $\mu^*$ with respect to the measure $m$ on $E$ is
\begin{equation}
    \mu^*(x) := (C^{*})^{-1}[0 \vee (\ubar\kappa - \tilde \kappa(x))]\pi(x),
    \label{eq:mu_star}
\end{equation}
where
$$
    C^*:= \int_{E} [0 \vee (\ubar\kappa - \tilde \kappa(x))]\pi(x)\dif m(x),
$$
assuming that this quantity is finite.

\begin{proposition}[Minimal regeneration distribution]
    Let $\mu^*, C^*$ be defined as above for some fixed $\ubar \kappa \ge 0$, where we assume $\mu^*$ is integrable and normalized. Let $\mu, C$ be any (normalized) probability measure on $E$ and positive constant respectively such that \eqref{eq:kap_nonneg} holds. Then $\mu^*$ minorizes $\mu$, in the sense that there exists some $\epsilon>0$ such that for all measurable $B \subset E$,
    \begin{equation}
        \mu(B) \ge \epsilon \mu^*(B),
        \label{eq:mu_star_minorize}
    \end{equation}
    and we have that
    \begin{equation*}
        C \ge C^*.
    \end{equation*}
    \label{prop:min_mu}
\end{proposition}
\begin{proof}
    From the assumption that \eqref{eq:kap_nonneg} holds, we must have that $\kappa \ge \kappa^*$ pointwise, from which it follows that for each $x \in E$,
    \begin{equation*}
        C \mu(x) \ge C^* \mu^*(x),
    \end{equation*}
    which establishes \eqref{eq:mu_star_minorize}, and by integrating both sides over $E$ it follows that $C \ge C^*$.
\end{proof}

How one can obtain samples from $\mu^*$ is in general not obvious, and is reminiscent of sampling from minorising measures as in \cite{Murdoch1998}. $\mu^*$ is generally compactly supported and supported around the modes of $\pi$; its support is contained within the set
$    \{x\in E: \tilde\kappa(x) < \ubar\kappa\}$,
and so often simulation is possible through straightforward rejection sampling.

On the other hand, the computation of $\kappa^*$ is immediate, since it does not require knowledge of $C^*$ or $\mu^*$ but is simply a thresholded version of $\tilde \kappa$ as in \eqref{eq:opt_kappa}.

When the interarrival process is already $\pi$-invariant, \textit{any} nonnegative value of $C$ can be used.
In this setting, the recent work of \cite{Caputo2019}, suggests that a sensible way to tune $C$ would be to choose it such that the average rate of regenerations matches the rate of mixing of the interarrival process. 
\cite{Caputo2019} showed that for the similar discrete-time PageRank surfer on random (finite) graphs, the resulting mixing time depends on the interplay between the rate of mixing of the underlying walk and the regeneration probability.

\subsection{Truncated regeneration rate} \label{s:truncated}

We consider now the diffusion case, as in Section \ref{subsec:symm_diff}. 
In this case $\kappa$ is typically unbounded, and the simulation of the lifetimes $\tau^{(i)}$ is not straightforward. In some cases using layered processes it is still possible to simulate $\taud$ exactly, as with 
the techniques of \cite{Pollock2016}. These are technically demanding, so in this section we consider the alternative of 
\textit{truncating} the regeneration rate. Namely, we fix some upper bound $M$, and work with the truncated regeneration rate 
$$
\kappa_M := \kappa \wedge M.
$$ 
This will introduce some approximation error, a discrepancy between the invariant distribution and $\pi$, but we will show how this error may be explicitly quantified.

In order to prove our result we will need to assume the following.\\

\noindent \textit{We assume that the interarrival process $Y$ is a diffusion on $\Rd$ satisfying BASSA, and that $\kappa$ is continuous. We also assume that Assumption \ref{assump:CFTP} holds, namely that we have a lower bound }
\begin{equation*}
    \kappa \ge \ubar\kappa >0.
\end{equation*}

\noindent Recall that under Assumption \ref{assump:CFTP}, $\taud$ can be stochastically dominated by an exponential random variable with rate $\ubar\kappa$, and hence all moments of $\taud$ are finite.

In order to avoid pathologies we assume that 
\begin{equation}
    M > \inf_{x \in E} \kappa(x) .
    \label{eq:M_lowerbd}
\end{equation}
We consider now the Restore process $X$ with interarrival process $Y$, regeneration density $\mu$ and truncated regeneration rate $\kappa_M$, for some given truncation level $M$ satisfying \eqref{eq:M_lowerbd}. 

Throughout this section we will be concerned only with the behavior of the Restore process \textit{before the first regeneration event}. As the regeneration distribution $\mu$ will not play a significant role we will consider the local process $Y$, without regenerations, and explicitly augment it with a first regeneration time. We will simply write $\E_x$ for the law of the local process $Y$ started from $x$, and consider the first arrival time $\taud$ to be a random variable defined by \eqref{eq:taud}.

Let us write $\kapMe$ for the excess regeneration rate over level $M$, that is,
\begin{equation*}
    \kappa_M^{\textrm e} := \kappa - \kappa_M.
\end{equation*}
Then by Poisson superposition, we can write 
\begin{equation}
    \taud = \tau_M \wedge \tau_M^\textrm e,
\label{eq:superpos}
\end{equation}
where $\taud, \tau_M, \tau_M^\mathrm e$ are the first arrival times of inhomogeneous Poisson process with rate functions $t\mapsto \kappa(Y_t)$, $t \mapsto \kappa_M(Y_t)$ and $t \mapsto \kappa_M^\mathrm e(Y_t)$ respectively, where these latter two Poisson processes are independent conditional on the path $t\mapsto Y_t$.

In particular, $\tau_M$ and $\tau_M^\mathrm e$ can be written as
\begin{equation}
    \tau_M = \inf\left \{t\ge 0: \int_0^t \kappa_M(Y_s)\dif s\ge \xi_1\right\},
    \label{eq:tau_M}
\end{equation}
\begin{equation}
    \tau_M^\mathrm e = \inf\left \{t\ge 0: \int_0^t \kappa_M^\mathrm e(Y_s)\dif s\ge \xi_2\right\},
    \label{eq:tau_Me}
\end{equation}
where $\xi_1, \xi_2 \sim \text{Exp}(1)$ are independent of each other and of the underlying process $Y$.

Since we are assuming Assumption~\ref{assump:CFTP} holds, by the arguments of Section~\ref{subsec:CFTP} it follows that the Restore process with regeneration rate $\kappa$ has a unique invariant distribution $\pi$, and from Section~\ref{subsec:CLT} the action of $\pi$ on a test function $f$ can be written as 
\begin{equation*}
    \nu_\pi[f] = \frac{\E_\mu\left[ \int_0^{\taud} f(Y_s) \dif s \right]}{\E_\mu[\taud]},
\end{equation*}
where here $Y$ is the local process without regenerations and $\taud$ is defined as in \eqref{eq:taud}. 

Similarly, the Restore process with truncated regeneration rate $\kappa_M$ is still uniformly ergodic and possesses a unique invariant distribution $\pi_M$.

Our goal now is to bound the total variation distance
\begin{equation*}
    \|\pi_M-\pi\|_1,
\end{equation*}
as a function of $M$.

\begin{theorem}
We have the following bound on the error.
    \begin{equation*}
        \|\pi_M - \pi\|_1 \le \frac{4 \int_0^\infty \P_\mu(\tau_M^\mathrm e \le t) \exp(-t\ubar\kappa)\dif t}{\E_\mu[\taud]}.
    \end{equation*}
    \label{Thm:bias_M}
\end{theorem}
\begin{proof}
    See Appendix~\ref{appen:other_proofs}.
\end{proof}

\begin{remark}
    To use this bound we need to further bound 
    \begin{equation*}
    \P_x(\tau_M^\mathrm e\le t).
\end{equation*}
Intuitively, if $\kappa_M$ is a reasonable approximation for $\kappa$, then $\kappa_M^\mathrm e$ is low, and hence $\tau_M^\mathrm e$ tends to be large, and so this bound is tighter.
\end{remark}

\begin{proposition}
    Fix a regeneration distribution $\mu$. We have that 
    \begin{equation}
        \int_0^\infty\dif t \, \P_\mu (\tau_M^\mathrm e \le t)\,\mathrm e^{-t\ubar\kappa}\to 0 \text{ as } M \to\infty.
    \label{eq:bartaud_to_0}
    \end{equation}
        Thus by Theorem \ref{Thm:bias_M} as $M\to \infty$,
    \begin{equation*}
        \|\pi_M-\pi\|_1 \to 0.
    \end{equation*}
    \label{prop:trunc_bd}
\end{proposition}
\begin{proof}
The event $\left\{ \tau_M^\mathrm{e} \le t \right\}$ is contained
in the event $\left\{ \sup_{s\le t} \kappa(Y_s) \ge M \right\}$.
Thus, for any fixed $x$
\begin{align*}
    \lim_{M\to\infty} \P_x (\tau_M^\mathrm e \le t) & =
      \P_x \left( \bigcap_{M=1}^\infty \left\{ \tau_M^\mathrm{e} \le t \right\}\right)\\
      &\le \P_x \left( \sup_{s\le t} \kappa(Y_s)=\infty \right)\\
      &\le \P_x \left( \sup_{s\le t} \| Y_s\|=\infty \right) \text{
       since $\kappa$ is locally bounded}\\
      &=0 \text{ since $Y$ is nonexplosive}.
\end{align*}
By the Dominated Convergence Theorem it follows that
$$
    \lim_{M\to\infty} \int_0^\infty \dif t \, \mathrm e^{-t\ubar\kappa} \int_E \dif\mu(x) \P_x 
    \left(\tau_M^\mathrm e \le t \right) =0,
$$
which is precisely \eqref{eq:bartaud_to_0}.
\end{proof}

In order for Theorem \ref{Thm:bias_M} to be of practical use, we will further need bounds on 
    \begin{equation}
    \P_x(\tau_M^\mathrm e\le t),
    \label{eq:tau_ME_probs}
\end{equation}
which will vary given the particular situation; given the choice of the underlying diffusion $Y$, target $\pi$ and regeneration density $\mu$. 

The rate at which the probabilities \eqref{eq:tau_ME_probs} decay as a function of $M$ will crucially depend on the rate at which the regeneration rate $\kappa$ grows. Thus we define the following,
\begin{equation*}
    L(M) := \sup\bigl\{\ell >0: \sup\{ \kappa(x)\,:\,  x\in [-\ell,\ell]^d \} \le M \bigr\},
\end{equation*}
which for a given truncation level $M$ defines the largest hypercube on which no truncation occurs.

The rate at which $L(M)$ grows as $M\to\infty$ will crucially dictate the rate at which the error decays. Then let
\begin{equation*}
    H(M) := [-L(M), L(M)]^d \subset\Rd,
\end{equation*}
and let
\begin{equation*}
    T_{M} := \inf \{t\ge 0: Y_t \in  \Rd \backslash H(M)\}
\end{equation*}
be the first hitting time of the diffusion $Y$ (without regenerations) of the complement of $H(M)$. Clearly we must have
\begin{equation*}
    T_M \le \tau_M^\mathrm e.
\end{equation*}
Thus it follows that
\begin{equation*}
    \int_0^\infty \P_x (\tau_M^\mathrm e \le t) \, e^{-\ubar \kappa t}\dif t \le \int_0^\infty \P_x(T_{M} \le t)\,e^{-\ubar \kappa t}\dif t .
\end{equation*}

To proceed from here we require knowledge of the distribution of the hitting times $T_M$ for the underlying diffusion $Y$. At this point we will \textit{specialize to the case of Brownian motion}; however, a similar analysis can be performed in any situation where we have analogous bounds on the hitting times.

By the reflection principle for one-dimensional Brownian motion we know that for any $a>0$,
\begin{equation*}
    \begin{split}
            \P\left(\sup_{0\le s\le t}|B_s|>a\right)&\le 2 \P\left (\sup_{0\le s\le t} B_s >a\right )=4\P(B_t >a)\\
            &= 4 \left(1-\Phi\left (\frac {a}{\sqrt t}\right)\right).
    \end{split}
\end{equation*}
Here $\Phi$ denotes the standard univariate normal cumulative distribution function. For a multidimensional standard Brownian motion, it follows that
\begin{equation*}
    \P_0\left(T_M\le t \right) \le 4d \left(1-\Phi\left (\frac {L(M)}{\sqrt t}\right)\right).
\end{equation*}
This is because leaving a hypercube is the same as having some component leaving the interval $[-L(M), L(M)]$.

We now make use of the well-known bound for the standard normal cumulative distribution function: for each $\lambda >0$,
\begin{equation*}
    1 - \Phi(\lambda) < \frac{1}{\sqrt{2\pi}\lambda} e^{-\lambda^2/2}.
\end{equation*}
This leads to the bound
\begin{equation*}
    \int_0^\infty \P_0(T_M \le t)\,e^{-\ubar \kappa t}\dif t \le \frac {4d}{\sqrt {2\pi}} \int_0^\infty \frac{\sqrt t}{L(M)} e^{-{L(M)^2}/(2t)} e^{-\ubar\kappa t}\dif t.
\end{equation*}
This integral can be evaluated analytically\footnote{\url{https://www.wolframalpha.com/input/?i=int_0\%5Einfty+\%5Csqrt+(t)+exp(-a\%2F(2t))+exp(-t)dt}}, to obtain
\begin{align*}
    \int_0^\infty \P_0(T_M \le t)\,e^{-\ubar \kappa t}\dif t &\le \frac{4d}{\sqrt{2\pi}}\frac{\sqrt \pi}{2\ubar\kappa} \left( \sqrt 2 +\frac{1}{L(M)\sqrt {\ubar\kappa}}\right) e^{-\sqrt{2\ubar\kappa} L(M)}\\
    &=\frac{2d}{\ubar\kappa}\left(1+\frac{1}{L(M) \sqrt{2\ubar\kappa}}\right) e^{-\sqrt{2\ubar\kappa} L(M)}.
\end{align*}
So for large values of $M$ we have a bound that decays like
\begin{equation*}
    e^{-\sqrt{2\ubar\kappa} L(M)}.
\end{equation*}
This can be used to give practical suggestions of how large to choose $M$ in order to balance the bias and variance of the algorithm's output.

Suppose we are able to obtain $n$ i.i.d. draws $X_1, \dots, X_n \sim \pi^M$, say by running the CFTP algorithm a total of $n$ times.
For a bounded test function $f$, we estimate $\nu_\pi[f]$ by 
\begin{equation*}
    \sum_{i=1}^n \frac{f(X_i)}{n}.
\end{equation*}
We estimate the error roughly as
\begin{align*}
    \bigg|\sum_{i=1}^n \frac{f(X_i)}{n} - \nu_\pi[f]\bigg |&\le \underbrace{\bigg |\sum_{i=1}^n \frac{f(X_i)}{n}- \pi^M(f)\bigg|}_{\sim \frac{1}{\sqrt n}}+\underbrace{|\pi^M(f)-\nu_\pi[f]|}_{\le \|f\|_\infty \|\pi^{M}-\pi\|_{\text{TV}}}\\
    &\approx O\bigg (\frac{1}{\sqrt n}\bigg ) + \exp\left(-\sqrt{2\ubar\kappa}L(M)\right).
\end{align*}

In order to balance these two terms, it is advisable to choose $n$ and $M$ such that
\begin{align*}
    \frac{1}{\sqrt n} &\sim \exp\left(-\sqrt{2\ubar \kappa} L(M)\right) \\
    \Rightarrow \frac{\log n}{ 2\sqrt {2\ubar\kappa}} &\sim L(M).
\end{align*}
So this gives some indication of how to choose $M$, given $n$. This will achieve an error of order roughly $O(n^{-1/2})$. The computational 
cost in $n$ will be roughly $O(n\log n)$.

\section{Examples} 
\label{sec:Restore_examples}
In this section we give some univariate examples which highlight key aspects of our Restore methodology. 
A thorough investigation of the computational properties of Restore is an important and challenging task, which is outside the scope of this present work and will be the topic of future research.

\subsection{Cauchy posterior}
We first give an example where $\pi$ has heavy tails and is multimodal, where we can apply coupling from the past.

This example is based on Example 3.1 of \cite{Murdoch2000}. We take 
\begin{equation}
    \bar\pi(x) \propto \prod_{i=1}^n \frac{1}{1+(y_i-x)^2},
    \label{eq:pi_multim}
\end{equation}
for some observations $(y_1, \dots, y_n) \in \R^n$, with respect to Lebesgue measure on $\R$. (We use the notation $\bar \pi$, since in the notation of Section~\ref{subsec:symm_diff}, the symbol $\pi$ is reserved for the target density with respect to the measure $\Gamma$.)

This can be thought of as the posterior distribution for i.i.d. Cauchy($x$) data, with an improper uniform prior on $\R$ for $x$. In Example 3.1 of \cite{Murdoch2000}, the author considers a very similar target with lighter tails. 
We will take the same data as \cite{Murdoch2000}, namely $n=3$ and observations $(1.3, -11.6, 4.4)$. The resulting posterior is plotted in red in Figure~\ref{fig:murdoc_post_samples}. 
Our sampling approach here is similar to that of \cite{Murdoch2000}; we are also combining local and global dynamics, but we will choose diffusive local dynamics which rapidly enter the tails.

As such, for our underlying process, we will take the following diffusion: an \textit{unstable} Ornstein--Uhlenbeck process, described by the SDE
\begin{equation}
    \dif Y_t = Y_t \dif t + \dif B_t,
    \label{eq:SDE_out_OU}
\end{equation}
where $B$ is a standard univariate Brownian motion. We showed in Section~\ref{subsubsec:BASSA_ex} that this diffusion satisfies the BASSA conditions (Assumption \ref{assump:symmetric_markov_bassa}). This diffusion, like a stable Ornstein--Uhlenbeck process, is also a Gaussian process with known finite-dimensional distributions, and so can be simulated easily without error.

For the regeneration distribution we will take the minimal regeneration distribution $\mu^*$ from Section~\ref{subsec:minimal_mu}, with $\ubar \kappa =4$. This distribution is compactly supported, and samples can be efficiently obtained through rejection sampling from a uniform distribution.
In this setting the regeneration rate is uniformly bounded from above, and so we can directly make use of Poisson thinning, as in Algorithm~\ref{alg:bdd_CFTP}. 
The various assumptions as in Section~\ref{subsec:symm_diff} are easily verified to hold in this setting.

Thus we are able to apply the CFTP implementation (Section~\ref{subsec:CFTP}) to obtain independent and identically distributed draws from $\pi$. 
A histogram consisting of 30,000 draws from the CFTP implementation are plotted in Figure~\ref{fig:murdoc_post_samples}. 
These were obtained by running the CFTP algorithm 30,000 times independently.

\begin{figure}
    \centering
    \includegraphics[width=0.8\textwidth]{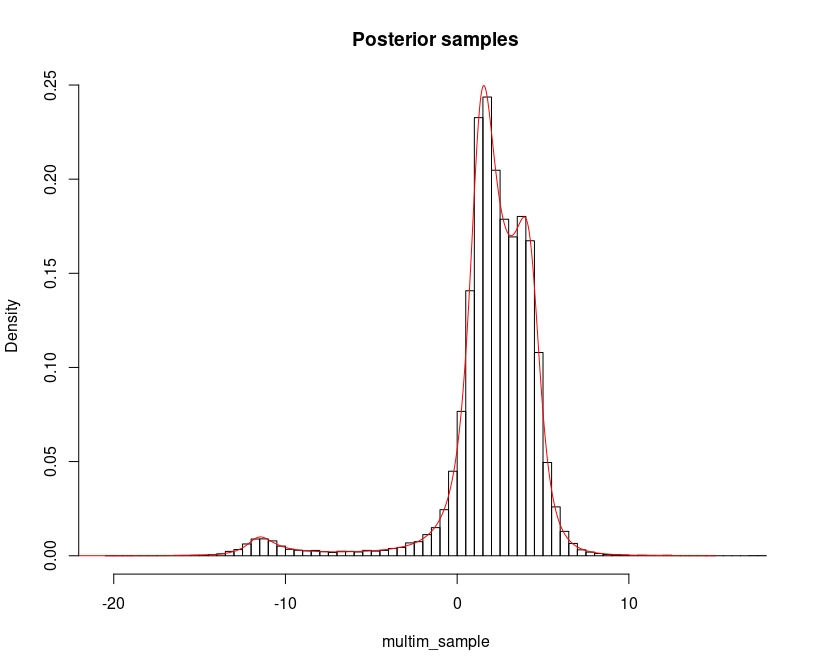}
    \caption{The heavy-tailed multi-modal target distribution $\pi$ of \eqref{eq:pi_multim} (red), and 30,000 samples obtained from the CFTP implementation. These are i.i.d. draws from $\pi$.}
    \label{fig:murdoc_post_samples}
\end{figure}

\subsection{Jump Restore example}
\label{subsubsec:jump_restore}
We turn now to an example of jump process Restore (Section~\ref{subsec:jump_proc}), where we use Restore to introduce rejection-free moves into an existing sampler.

A situation where Assumption \ref{assump:MCMC} is easily checked is when $P$ corresponds to a Markov chain that is already $\pi$-invariant, for instance the kernel of an appropriate MCMC algorithm targeting $\pi$.
In this case we can easily embed $P$ into continuous time without changing the asymptotic dynamics, 
just by taking constant holding rates $\lambda \equiv 1$.
In this case the regeneration rate reduces to
\begin{equation*}
    \kappa(x) = C \frac{\mu(x)}{\pi(x)}, \quad x \in E,
\end{equation*}
and we see that any choice of $C>0$ will ensure nonnegativity of $\kappa$. This gives a recipe to introduce rejection-free moves to a discrete sampler in continuous time.

Consider the following example, in one dimension for ease of visualisation. Writing $\phi(\cdot; \nu, \sigma^2)$ for the univariate Gaussian density with mean $\nu \in \R$ and variance $\sigma^2 >0$, take as the target $\pi$ on $\R$:
\begin{equation*}
    \pi(x) = 0.1\, \phi(x; -22, 3^2) + 0.3\, \phi(x; -1, 0.2^2) + 0.6\, \phi(x; 15, 1^2), \quad x \in \R.
\end{equation*}
For the regeneration density $\mu$, we take
\begin{equation*}
    \mu(x) = \frac{1}{3} \left( \phi(x; -29, 0.3^2)+ \phi(x;3, 1^2) + \phi(x; 10, 1^2) \right), \quad x \in \R.
\end{equation*}

We take the underlying process $Y$ to be Random Walk Metropolis with variance 1 embedded in continuous time, with constant holding rate 1. We took the constant $C=1$ in the regeneration rate.

We have plotted a histogram after 300,000 steps of the jump chain (taking into account holding times) in Figure~\ref{fig:jump_rest_hist} and in Figure~\ref{fig:jump_rest_traj} we have plotted the continuous-time trajectory of the first 50,000 jump steps of this run.

\begin{figure}
    \centering
    \includegraphics[width=0.8\textwidth]{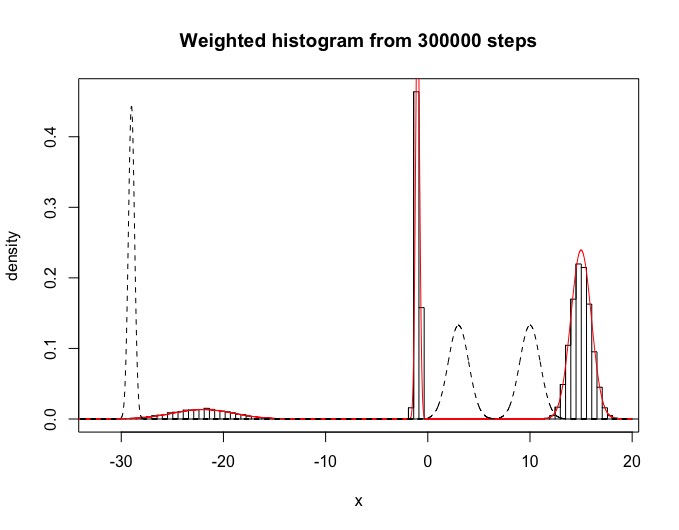}
    \caption{The target density $\pi$ (red) and regeneration density $\mu$ (dashed) for the jump Restore example, along with a weighted histogram of the Restore run, taking into account holding times.}
    \label{fig:jump_rest_hist}
\end{figure}

\begin{figure}
    \centering
    \includegraphics[width=0.8\textwidth]{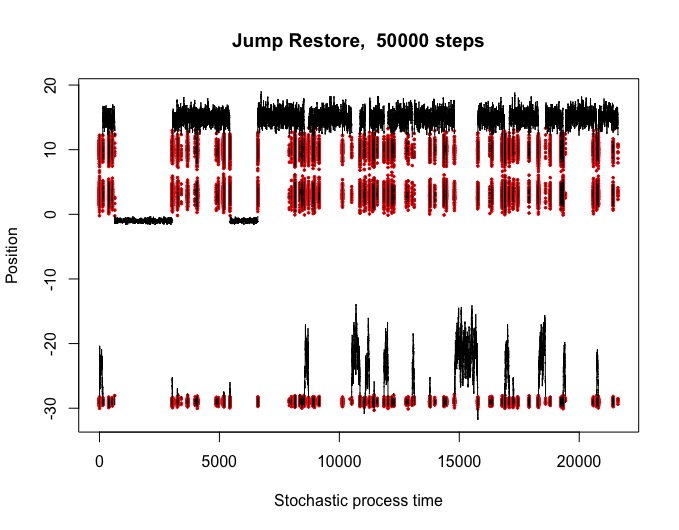}
    \caption{The continuous-time trajectory of the first 50,000 steps of the jump Restore sampler. The proportion of regeneration moves was roughly $0.496$. The red points are the regenerations.}
    \label{fig:jump_rest_traj}
\end{figure}


\section{Conclusions}
\label{sec:Restore_conclusions}
In this work we have introduced and studied the {\em Restore process}, which is obtained by enriching an existing {\em local} continuous-time Markov process with {\em global} regenerative dynamics. We have focused particularly on how it could be applied within Monte Carlo methodology to sample from a prescribed target density of interest, $\pi$. Surprisingly, the Restore process enables us to combine continuous-time local and global dynamics -- \textit{neither of which is $\pi$-invariant} -- and by means of an inhomogeneous Poisson process compensate these dynamics to ensure the process is $\pi$-stationary. The resulting sampler is simple to implement (in many settings it is no more complex than a vanilla Metropolis--Hastings sampler), and we readily establish a central limit theorem.  Although the use of an inhomogeneous Poisson process has some natural affinity with recent Monte Carlo developments (such as piecewise-deterministic MCMC methods, \cite{Bouchard-Cote2018, Bierkens2019} and quasi-stationary Monte Carlo methods, \cite{Pollock2016, Wang2019}), the additional regenerative behavior 
provides significant benefits, both theoretically and practically. In particular, we show that the regenerative behavior allows us, with verifiable conditions, to demonstrate uniform ergodicity, avoid traditional MCMC problems such as burn-in, and even construct independent exact draws from $\pi$ by a
coupling-from-the-past scheme. On the more practical side, 
we have discussed some of the natural approximations a practitioner may make in implementing continuous-time samplers for use within the Restore process, showing that the global regenerative dynamics allow us to readily analyze such approximations, and to study and understand their effect. 

This work has been primarily focused on introducing and establishing fundamental results for the Restore process, and there is considerable scope for further development. From a practical perspective, it is of interest to consider its use in different settings, with new classes of local and global dynamics. For instance, in the setting of continuous-time local dynamics one could consider piecewise-deterministic Markov processes, or even ODE flows, which will have properties particularly suited to certain problems. 
Or the construction of the global dynamics --- the regeneration density --- might
utilize \textit{other} inference about the target density $\pi$, obtained, say, by a simpler but approximate scheme. Thus in effect one could use Restore to remove the error from existing statistical approaches. 
Indeed, the flexibility offered by Restore, in combining continuous-time local and global dynamics which are not independently $\pi$-invariant, suggests that we are still far from espying the limits to which this general framework might be extended. 
For instance, it may be possible to use the framework to compensate multiple local and global dynamics, or to have global dynamics which are themselves adaptive to the accrued information of the process. Other prospective applications of the Restore process include exploiting the regenerative structure of the process for use within parallel computing architectures, embedding Restore 
within other Monte Carlo methodologies which require independent
exact draws from the target distribution, something that
Restore (unlike MCMC) can provide.

From a theoretical perspective, our understanding of the Restore
process is nowhere near as complete as we should like, 
beginning with the lack of a single unified proof of invariance
of the target distribution, and proceeding to the need for 
an appropriate definition of `efficiency' or `optimality' for choosing appropriate dynamics. Any reckoning with these notions
would have to take account of the temporally varying computational
cost of simulating the process with multiple dynamics, which is beyond the scope of this paper. 
Insights into issues such as these could be particularly useful in the design of appropriate diagnostics for the process, for instance in situations where the verifiable CFTP conditions do not hold.

\section{Acknowledgements}
We would like to thank Martin Kolb for supplying arguments pertaining to domains of self-adjoint operators. We would also like to thank Jere Koskela, Divakar Kumar, Hector McKimm and Abraham Ng for interesting discussions on aspects of this paper. We would like to thank the associate editor and anonymous referee for their comments which have substantially improved the paper.

Research of A.~Q.~Wang is supported by the EPSRC CoSInES (COmputational Statistical INference for Engineering and Security) grant EP/R034710/1 and the OxWaSP CDT through grant EP/L016710/1. Most of his work was conducted as a doctoral student at the Department of Statistics, University of Oxford, and a preliminary version of this work appears in his doctoral thesis, \cite{Wang2020}.
Research of M.~Pollock is supported by The Alan Turing Institute under the EPSRC grant EP/N510129/1, and his contribution was largely conducted while at the University of Warwick.
Research of G.~O.~Roberts is supported by EPSRC grants
EP/R034710/1, EP/R018561/1, EP/K014463/1. 
Research of D.~Steinsaltz is supported by BBSRC grant
BB/S001824/1.

\begin{appendices}
\section{Proofs}

\subsection{Proof of Lemma~\ref{prop:Restore_exist}}
\label{subsec:pf_exist}
    The techniques of Chapter 61 of \cite{Sharpe1988} allow us to identify the process $Y$ killed at time \eqref{eq:taud} with the (sub-)process generated by the decreasing multiplicative functional
    \begin{equation*}
        m_t := \exp\left(-\int_0^t \kappa(Y_s) \dif s \right), \quad t \ge 0.
    \end{equation*}
    Since $Y$ is right-continuous and $\kappa$ is locally bounded, this defines a \textit{right} multiplicative functional. Then by Theorem~61.5 of \cite{Sharpe1988}, concerning processes which are generated by such right multiplicative functionals, we can conclude that our killed process $Y$ with lifetime \eqref{eq:taud} is a right process. 
    
    The resulting Restore process $X$, given in \eqref{eq:Restore_process}, is formed by concatenating independent copies of such killed processes with initial distribution $\mu$.  Exercise~14.17 of \cite{Sharpe1988} shows that the infinite concatenation of a series of independent and identically distributed right processes is yet another right processes. Hence our Restore process $X$ is indeed a right process.
    
     Since $\kappa$ is locally bounded and $Y$ is right-continuous, it follows that $\E_\mu[\taud]>0$.
    The final statement then follows from the fact that the lifetimes (apart from possibly the first) are independent and identically distributed.

\subsection{Proof of Proposition~\ref{prop:generator}}
\label{sec:proof_generator}
We have seen in the Proof of Lemma~\ref{prop:Restore_exist} that the killed process can be identified with a subprocess generated by a multiplicative functional. Hence we will seek to utilize Theorem~2.5 of \cite{Demuth2000}, which can be applied to such processes. From our Assumptions~\ref{assump:pi} and \ref{assump:kappa}, it follows that $\kappa$ is nonnegative and continuous. In particular it is bounded on compact sets. Thus the potential $V=\kappa$ belongs to the so-called \textit{local Kato class}. 
We are assuming that our underlying process $Y$ satisfies BASSA (Assumption \ref{assump:symmetric_markov_bassa}), so we have satisfied the conditions of Theorem~2.5 of \cite{Demuth2000}. 

The conclusions of the Proposition are precisely the conclusions of Theorem~2.5 of \cite{Demuth2000}, restated in our present setting.

\subsection{Proof of Lemma~\ref{lemma:L1_pi_cond}} \label{sec:proof_L1_pi_cond}
Recall that we write $W^{2,1}(\R^d)$ for the Sobolev space of measurable functions on $\R^d$ whose first and second derivatives are integrable with respect to Lebesgue measure on $\R^d$, equipped with the corresponding Sobolev norm. Precise definitions can be found in \cite{Adams1973}.

First note that the fact that $\bar\pi \in W^{2,1}(\R^d)$ along with the integral assumptions imply that
\begin{equation*}
    \int_{\R^d} \pi(x)\kappa(x) \dif \Gamma(x) <\infty,
\end{equation*}
since we can write
\begin{equation*}
    \pi \kappa\gamma = \frac{1}{2}\Delta \bar\pi - \nabla A\cdot \nabla \bar\pi - \Delta A \bar\pi + C\mu \gamma.
\end{equation*}

Now, as in the proof of Theorem~3.18 of \cite{Adams1973}, let $f:\Rd\to \R$ be a mollifier, satisfying properties (i), (ii) and (iii) described therein with ${m=2}$. Taking the square if necessary, we can assume that $f$ is nonnegative. In particular, $f$ and its derivatives up to order 2 are bounded pointwise in absolute value by a constant $M$.
We can now define similarly for each $n\in \mathbb N$,
$  \pi_n := f_n\pi,$
where $f_n(x) := f(x/n), x\in \R^d$.

Since $\pi$ is smooth, the $\pi_n$ are a sequence of smooth, compactly supported functions with the following properties: $\bar\pi_n := \gamma \cdot \pi_n$ converges to $\bar\pi = \gamma\cdot \pi$ pointwise and in $W^{2,1}(\R^d)$ (as in 
the proof of Theorem~3.18, \cite{Adams1973}), and we have $\pi_n \le M \pi$ pointwise, uniformly over $n$.
This implies, in particular, that $\pi_n$ converges to $\pi$ in $\L^1(\Gamma)$.

Using the relation $\nabla A = \frac{\nabla \gamma}{2\gamma}$ we can relate the action of
$L^0$ on $\Gamma$-densities to its action on Lebesgue densities:
\begin{equation} \label{eq:piLebesgue}
    -\gamma L^0 \pi_n =\frac12 \Delta \pi_n +\nabla A \cdot \nabla\pi_n = \frac{1}{2}\Delta \bar \pi_n - \nabla A \cdot \nabla \bar \pi_n - (\Delta A)\, \bar\pi_n.
\end{equation}
This will allow us to show that $-\gamma L^0 \pi_n$ converges in $\L^1(\R^d)$ --- $\R^d$ equipped with Lebesgue measure --- to
\begin{equation*}
    \frac{1}{2}\Delta \bar \pi - \nabla A \cdot \nabla \bar \pi - \Delta A\, \bar \pi.
\end{equation*}
We consider the three terms on the right-hand side individually.
Convergence of the first term is immediate since $\bar\pi_n$ converges to $\bar \pi$ in $W^{2,1}(\R)$. 
Convergence of the third term follows since $\bar \pi$ converges to $\bar \pi$ pointwise, with $\pi_n \le M \pi$, so we can make use of the Dominated Convergence Theorem. 

It remains to demonstrate the convergence of $\nabla A \cdot \nabla \bar\pi_n$.
We have
\begin{equation}
    \nabla A \cdot \nabla \bar \pi_n = f_n \nabla A \cdot \nabla \bar\pi+ \bar \pi \nabla A \cdot \nabla f_n.
    \label{eq:lemma4_pf}
\end{equation}
The first term on the right-hand side converges in $\mathcal L^1(\R^d)$ to $\nabla A \cdot \nabla \pi$ 
straightforwardly, by dominated convergence, as the mollifiers are uniformly bounded by $M$.
For the second term of \eqref{eq:lemma4_pf}, first note that $\nabla f_n(x) = n^{-1} \nabla f (x/n)$, which is bounded (in each component) by $M/n$, and the support of $\nabla f_n$ is by construction within the set $B_n :=\{ y \in \R^d: n\le |y|\le 2n\}$. Thus we have 
\begin{equation*}
    |\bar \pi(x) \nabla A \cdot \nabla f_n(x)| \le \bar\pi(x) K |x| M n^{-1} \,1_{B_n}(x) \le 2KM\bar \pi(x).
\end{equation*}
Here we used the bound on the drift $|\nabla A(x)|\le K|x|$. Thus we can apply
the Dominated Convergence Theorem once more to establish that $\bar\pi \nabla A \cdot \nabla f_n$ is converging in $\mathcal L^1(\R^d)$ to the zero function.

It follows (by reversing the application of \eqref{eq:piLebesgue}) that $-L^0 \pi_n$ converges to 
\begin{equation*}
    \frac{1}{2}\Delta \bar \pi - \nabla A \cdot \nabla \bar \pi - (\Delta A)\, \bar\pi = \frac{1}{2}\Delta \pi + \nabla A \cdot \nabla \pi
\end{equation*}
in $\L^1(\Gamma)$. 
Since $L^0$ is a closed operator, and since the sequence of smooth 
compactly supported functions $\pi_n$ belongs to $\Dom(L^0_1)$ for each $n$, we have thus established that $\pi \in \Dom(L^0_1)$.

Finally, since $\int \pi\kappa \dif \Gamma<\infty$, and we have that $\pi_n \le M \pi$, we have that $\int \pi_n \kappa \dif \Gamma \to \int \pi \kappa\dif \Gamma$.

Thus 
\begin{equation*}
    L^\kappa \pi_n \to L^0 \pi + \kappa \pi = C \mu
\end{equation*}
in $\L^1(\Gamma)$. This shows that $\pi \in \Dom(L^\kappa_1)$ and $L^\kappa_1 \pi = C\mu$, concluding the proof of Lemma~\ref{lemma:L1_pi_cond}.

\subsection{Proof of Theorem~\ref{thm:pi_inv_symm}} \label{sec:proof_restore_invariant_diffusion}
We want to prove that $\pi$ is an invariant distribution for the Restore process $X$ with interarrival dynamics $Y$, regeneration rate $\kappa$ as defined in \eqref{eq:kappa_symm} with regeneration density $\mu$. We are in the setting $(E,m) = (\Rd, \Gamma)$.

We know that the Restore process $X$, formed by concatenating copies of the killed process, is a strong Markov process; see Lemma~\ref{prop:Restore_exist} and its proof. Let $\{P^\mu_t:t\ge 0\}$ denotes its semigroup.

Our goal is to show 
$$
    t\mapsto \pi P_t^\mu f := \int \dif \Gamma(x) \pi(x) [P_t^\mu f](x)=\E_\pi[f(X_t)]
$$ 
is constant in $t$. By time homogeneity it suffices to show that the time-derivative 
is 0 at $t=0$. This is the same method used to prove $\pi$-invariance of the Bouncy Particle Sampler in the supplementary material of \cite{Bouchard-Cote2018}.

The Restore process naturally exhibits renewal behavior, since the individual lifetimes
are independent and identically distributed. So we will seek a renewal-type representation of the semigroup $P_t^\mu$ by conditioning on the first arrival $\taud$. Since $\kappa$ is locally bounded, $\taud$ is absolutely continuous on $\R^+$, hence will possess a density with respect to Lebesgue measure on $\R^+$.

Since $\kappa$ is nonnegative, the semigroup $\exp(-t\Lk)$ can also be expressed as
\begin{equation*}
    [\exp(-t\Lk)f](x) = \E_x\left[f(Y_t) 1\{\taud > t\} \right],
\end{equation*}
where $\taud$ is defined as in \eqref{eq:taud} for each $f$ where the integral is well-defined.

Note that we have
\begin{equation}
    -\Lk \pi = -L^0 \pi - \kappa \pi =  -C \mu.
    \label{eq:Lkpi}
\end{equation}
This equation holds formally, where we view $\Lk$ and $L^0$ as formal differential operators, and
as a statement about the $\L^1(\Gamma)$ generator by Assumption~\ref{assump:technical}. 
Since we additionally assume that in Assumption~\ref{assump:technical} that both $\pi$ and $\mu$ 
are in $\L^2(\Gamma)$ it follows that \eqref{eq:Lkpi} also holds for the $\L^2(\Gamma)$ generator as well.

Consider 
\begin{align*}
    \P_\pi (\taud >t) &= \int \dif \Gamma(x) \pi(x) \int \dif \Gamma(y) \,p^\kappa(t,x,y)\\
    &= \int \dif \Gamma(y) \int \dif \Gamma(x) \, \pi(x) p^\kappa(t,y,x) \\
    &= \int \dif \Gamma(y) [\exp(-t\Lk)\pi](y),
\end{align*}
where the second line relies on Tonelli's theorem to exchange the order of integration, 
and uses the symmetry of $p^\kappa$ to replace $p^\kappa(t,x,y)$ by $p^\kappa(t,y,x)$. 
The final integral is well-defined since $\pi \in \L^1(\Gamma)$ and the semigroup $\etLk$ maps $\L^1(\Gamma)$ to itself, by Proposition~\ref{prop:generator}. Thus by strong continuity and the fact that $\pi \in \Dom(L^\kappa_1)$ (Assumption \ref{assump:technical}) we can differentiate this expression to find
\begin{align*}
    \frac{\dif \phantom{t}}{\dif t}\P_\pi(\taud > t)|_{t=s} &= \int \dif \Gamma(y) [\esLk (-\Lk\pi)](y) \\
    &= -\int \dif \Gamma(y) [\esLk (C\mu)](y)\\
    &= -C \int \dif \Gamma(x) \mu(x) [\esLk 1](x) .
\end{align*}
The second line applies \eqref{eq:Lkpi} again, 
while the final equality relies once more
on Assumption \ref{assump:technical} and symmetry of the semigroup.

This shows that the density on $\R^+$ with respect to Lebesgue measure of the first arrival time under $\P_\pi$ is given by
\begin{equation*}
    h(s) = C \int \dif \Gamma(x) \mu(x) [\exp(-s\Lk) 1](x) =  C\, \P_\mu(\taud > s) ,\quad s \ge 0,
\end{equation*}
and that $C = 1/\E_\mu[\taud]$.
This allows us to represent the semigroup of the Restore process started in $\pi$ as
\begin{align*}
    \pi P_t^\mu f &= \int_0^t  C 
    \P_\mu(\taud > s)\,
    \mu P_{t-s}^\mu f\dif s + \pi \etLk f \\
    &= C \int_0^t  
    \P_\mu(\taud > t-s)\,
    \mu P_s^\mu f \dif s+ \pi \etLk f.
\end{align*}
Our goal is to differentiate this expression with respect to $t$, and to show that the derivative at $t=0$ is zero. 

Consider any bounded $f$ in $\Dom(L^\kappa_2)$. From the representation above we can see that $t\mapsto \pi P_t^\mu f$ is a continuous function.
Starting from 
\begin{equation*}
\begin{split}
        t \mapsto \P_\mu(\taud > t) &= \int \dif \Gamma(x) \mu(x)\int p^\kappa(t,x,y) \dif \Gamma(y) \\
        &= \E_\mu^0 \left[\exp\left( -\int_0^t \kappa(Y_s)\dif s \right) \right],
\end{split}
\end{equation*}
our technical assumption \eqref{eq:tech_cond_mu} allow us to differentiate under the integral sign to obtain
\begin{equation*}
    g(s):= -\frac{\dif\phantom{t}}{\dif t} \P_\mu(\taud>t)|_{t=s} = \E_\mu^0\left[ \kappa(Y_s) \exp\left( -\int_0^s \kappa(Y_u)\dif u \right) \right]
\end{equation*}
for each $s \in [0,1]$. $g$ is a continuous function, and will be uniformly bounded over $s \in [0,1]$.

Conditioning, as above, on the first regeneration time, we then have
\begin{equation*}
    \mu P_t^\mu f = \int_0^t  
    g(s)\,
    \mu P^\mu_{t-s} f \dif s 
    + \mu \etLk f,
\end{equation*}
showing that $t\mapsto \mu P^\mu_t f$ is also a continuous function.
By Leibniz's rule:
\begin{align*}
    \frac{\dif }{\dif t} \pi P_t^\mu f \bigg |_{t=s}=  C\int_0^s  
    g(s-u)\,\mu P_u^\mu f 
    \dif u 
    +  C \mu(1) \mu P_s^\mu f + \pi \exp(-t\Lk) (-L^\kappa f). 
\end{align*}
Taking $t=0$, we find
\begin{equation*}
    \frac{\dif }{\dif t} \pi P_t^\mu f \bigg |_{t=0} = C \mu(f) + \pi (-L^\kappa f).
\end{equation*}
Since we chose $f \in \Dom(L^\kappa_2)$, and since we assumed $\pi \in \Dom (L^0_2)\cap \Dom(L_2^\kappa)$ (Assumption~\ref{assump:pi}), this final expression is equal to
\begin{equation*}
    C \mu(f) + \pi(-L^\kappa f)=\int \dif \Gamma(x) \,f(x) \left ( C \mu(x)+ (-L^0\pi)(x) - \kappa(x)\pi(x) \right).
\end{equation*}
This will equal 0 for any such $f$ if
\begin{equation*}
    \kappa(x) = \frac{-L^0 \pi}{\pi}(x) + C\frac{\mu(x)}{\pi(x)}, \quad x \in \Rd,
\end{equation*}
which is exactly our \eqref{eq:kappa_symm}. This concludes the proof of Theorem~\ref{thm:pi_inv_symm}.


\subsection{Proof of Theorem~\ref{Thm:bias_M}}
\label{appen:other_proofs}
Recall the expression for the invariant distribution
\begin{equation*}
    \nu_\pi[f]= \frac{\E_\mu\left[ \int_0^{\taud} f(Y_s) \dif s \right]}{\E_\mu[\taud]}.
\end{equation*}
We can rewrite this by exchanging the order of integration. Consider the resolvent operator, which maps measurable functions to measurable functions,
\begin{align*}
    \Res f(x) :=& \int_0^\infty \dif t\, \E_x [f(Y_t)1\{\taud>t\}]\\
    =& \int_0^\infty \dif t \, \E_x \left[ f(Y_t) 1\{ \tau_M>t\}1\{\tau_M^\mathrm e>t\}\right],
\end{align*}
where the second equality holds by \eqref{eq:superpos}.
Note that given a bounded measurable function $f$, $\Res f$ is also a bounded measurable function, since we can bound
\begin{equation*}
    |\Res f(x)| \le \|f\|_\infty \E_x[\taud] \le \|f\|_\infty \,\ubar\kappa^{-1}.
\end{equation*}
Thus by Fubini's theorem, we can write
\begin{equation*}
    \pi = \frac{\mu \Res}{\mu \Res 1},
\end{equation*}
in analogue with expressions given in \cite{wang2019approx} and \cite{Benaim2018}.

The invariant distribution $\pi_M$ of the process with truncated rate can be represented in a similar way. Write
\begin{equation*}
    \pi_M = \frac {\mu \Res_M}{\mu \Res_M 1},
\end{equation*}
where for bounded measurable $f$,
\begin{equation*}
    \Res_M f(x) = \int_0^\infty \dif t\, \E_x [f(Y_t)1\{\tau_M>t\}].
\end{equation*}

    We have that
\begin{align*}
    \pi_M-\pi &= \frac{(\mu\Res1) \mu\Res_M - (\mu\Res_M 1)\mu\Res }{(\mu\Res_M 1)(\mu\Res1)}\\
    &= \frac{\mu\Res 1 (\mu\Res_M - \mu\Res) +(\mu\Res1 - \mu\Res_M 1)\mu\Res}{(\mu\Res_M 1)(\mu\Res1)}.
\end{align*}
So we would like to bound $$ |\mu\Res f - \mu\Res_M f|$$ for arbitrary bounded measurable $f$.

For any nonnegative bounded measurable $f$,
\begin{align*}
    |\mu\Res f - \mu\Res_M f|&\le  \int \mu(\dif x)\int \dif t\, |\E_x[f(Y_t)1\{\tau_M >t\}1\{\tau_M^\mathrm e>t\}]\\
    &\quad\quad -\E_x[f(X_t) 1\{\tau_M>t\}] |\\
    &\le \int \mu(\dif x) \int \dif t\, \E_x \left [f(Y_t) \left(1- 1\{\tau_M^\mathrm e >t\}\right )1\{\tau_M>t\}\right]\\
    &\le \|f\|_\infty \int \mu(\dif x) \int \dif t \,\E_x [1\{\tau_M^\mathrm e \le t, \tau_M>t\}]\\
    &= \|f\|_\infty \int \mu(\dif x)\int \dif t \, \P_x (\tau_M^\mathrm e \le t, \tau_M >t).
\end{align*}
Since we are assuming that we have a lower bound $\ubar \kappa$ on the regeneration rate, and Assumption \ref{eq:M_lowerbd} holds, we can stochastically bound $\tau_M \le \tau'$ where $\tau' \sim \text{Exp}(\ubar \kappa)$ and $\tau'$ is independent of everything else. So continuing the chain of inequalities,
\begin{align*}
    &\le \|f\|_\infty \int \mu(\dif x)\int \dif t \, \P_x (\tau_M^\mathrm e \le t, \tau' >t)\\
    &= \|f\|_\infty \int \mu(\dif x) \int \dif t \,\P_x(\tau_M^\mathrm e \le t)\, \P_x(\tau' > t)\\
    &=\|f\|_\infty \int \mu(\dif x)\int \dif t \,\P_x(\tau_M^\mathrm e\le t) \,\mathrm e^{-t\ubar \kappa}.
\end{align*}
A universal upper bound on this quantity is $\|f\|_\infty/\ubar \kappa$.

For for a given continuous nonnegative bounded $f$ with $\|f\|_\infty \le 1$ we get the following bounds.
\begin{align*}
    |&\pi_M f- \pi f|\\
    &\le \frac{\mu \Res 1\|f\|_\infty \int \mu(\dif x)\int \dif t \,\P_x(\tau_M^\mathrm e \le t) \,\mathrm e^{-t\ubar \kappa} + \int \mu(\dif x)\int \dif t \,\P_x(\tau_M^\mathrm e\le t) \,\mathrm e^{-t\ubar \kappa} |\mu \Res f|}{(\mu \Res_M 1)(\mu \Res 1)}\\
    &\le \frac{\int \mu(\dif x)\int \dif t \,\P_x(\tau_M^\mathrm e\le t) \,\mathrm e^{-t\ubar \kappa}}{\mu\Res_M 1} + 
    \frac{\int \mu(\dif x)\int \dif t \,\P_x(\tau_M^\mathrm e\le t) \,\mathrm e^{-t\ubar \kappa}}{\mu\Res_M1}\\
    &= \frac{2 \int \mu(\dif x)\int \dif t \,\P_x( \tau_M^\mathrm e\le t) \,\mathrm e^{-t\ubar \kappa}}{\mu\Res_M 1}\\
    &\le \frac{2 \int \mu(\dif x)\int \dif t \,\P_x(\tau_M^\mathrm e\le t) \,\mathrm e^{-t\ubar \kappa}}{\mu\Res1}\\
\end{align*}
Since this bound is valid for only nonnegative bounded $f$, in order to bound $\|\pi_M - \pi\|_1$ we pick up an additional factor of 2.

This concludes the proof of Theorem~\ref{Thm:bias_M}.

\end{appendices}

\bibliography{Restore_bib2}
\bibliographystyle{plainnat}

\end{document}